\documentclass[11pt]{amsart}

\usepackage{amssymb,amsfonts,amsthm,amsmath}

\usepackage{booktabs}

\usepackage{graphicx, verbatim, centernot}
\usepackage{xcolor}

\usepackage{multicol}
\usepackage{enumitem}

\usepackage[a4paper, hmargin=2cm, vmargin={3cm, 3cm}]{geometry}
\usepackage{fancyhdr}
\usepackage{hyperref}
\hypersetup{
	colorlinks=true,
	linkcolor=blue,
	citecolor=purple,
	filecolor=magenta,      
	urlcolor=purple,
}

\usepackage[
backend=bibtex,
style=numeric,
sortlocale=auto,
uniquelist=false,
natbib=false,
url=false, 
doi=false,
eprint=false,
safeinputenc=false,
giveninits=true,
isbn=false,
maxbibnames=10,
]{biblatex}

\renewbibmacro{in:}{}

\DeclareFieldFormat*[inbook]{citetitle}{#1}
\DeclareFieldFormat*{title}{#1}

\newtheorem{theorem}{Theorem}
\newtheorem{lemma}[theorem]{Lemma}

\newtheorem{proposition}[theorem]{Proposition}

\newtheorem{claim}[theorem]{Claim}

\theoremstyle{remark}

\theoremstyle{definition}

\def\eps{{\varepsilon}}

\def\Ga{{\Gamma}}
\def\la{{\lambda}}

\newcommand{\bR}{\mathbb R}
\newcommand{\cR}{\mathcal R}
\newcommand{\bC}{\mathbb C}
\newcommand{\bZ}{\mathbb Z}
\newcommand{\bD}{\mathbb D}

\newcommand{\bQ}{\mathbb Q}

\newcommand{\bP}{\mathbb P}
\newcommand{\bS}{\mathbb S}
\newcommand{\bE}{\mathbb E}
\newcommand{\cE}{\mathcal{E}}

\newcommand{\cA}{\mathcal A}
\newcommand{\cB}{\mathcal B}
\renewcommand{\P}{\mathbb P}

\newcommand{\cC}{\mathcal C}
\newcommand{\cN}{\mathcal N}
\newcommand{\cG}{\mathcal G}

\newcommand{\Cov}{\operatorname{Cov}}
\newcommand{\Otilde}{\widetilde{O}}

\newcommand{\polylog}{\mathrm{polylog}}

\newcommand{\dd}{{\rm d}}
\newcommand{\I}{{\rm i}}
\newcommand{\disc}{\Delta}
\newcommand{\one}{\boldsymbol{1}}

\begin{document}

    \title[Law of large numbers for the discriminant of random polynomials]{Law of large numbers for the discriminant of random polynomials}

    \author[Marcus Michelen]{Marcus Michelen}
    \author[Oren Yakir]{Oren Yakir}
    \address{Department of Mathematics, Statistics, and Computer Science, University of Illinois, Chicago. }
    \email{michelen.math@gmail.com}
    \address{Department of mathematics, Stanford University.}
    \email{oren.yakir@gmail.com}

	\maketitle

	\begin{abstract}
    Let $f_n$ be a random polynomial of degree $n$, whose coefficients are independent and identically distributed random variables with mean-zero and variance one. Let $\Delta(f_n)$ denote the discriminant of $f_n$, that is $\Delta(f_n) = A^{2n-2}\prod_{i < j} (\alpha_j - \alpha_i)^2$ where $A$ is the leading coefficient of $f_n$ and $\alpha_1,\ldots\alpha_n$ are its roots.
    We prove that with high probability $$|\Delta(f_n)| = n^{2n} e^{-{\sf D}_\ast n(1+o(1))}$$ as $n\to \infty$, for some explicit universal
    constant ${\sf D}_\ast>0$. A key step in the proof is an analytic representation for the logarithm of the discriminant, which captures both the distributional reciprocal symmetry of the random roots and the cancellations this symmetry induces.
	\end{abstract}

    \pagestyle{plain}
    
	\section{Introduction}
	
    In this paper we study the discriminant of the random polynomial 
	\begin{equation}
		\label{eq:intro_def_of_kac_polynomial}
		f_n(z) = \sum_{k=0}^n \xi_k z^k 
	\end{equation}
	where $\xi_0,\ldots,\xi_n$ are i.i.d.\ random variables, with a common sub-Gaussian\footnote{We say that $\xi$ is sub-Gaussian if there exist $c>0$ such that 
	$
	\bP\big(|\xi|\ge t\big) \le 2\exp(-ct^2)
	$
	for all $t>0$. } 
    distribution $\xi$ which satisfies $$\bE[\xi] = 0 \qquad \text{and} \qquad \bE[|\xi|^2] = 1 \, .$$  
    Random polynomials sampled according to~\eqref{eq:intro_def_of_kac_polynomial} are usually referred to in the literature as \emph{Kac} polynomials. Recall that for a polynomial $P(z) = A\prod_{j=1}^n(z-\alpha_j)$ the \emph{discriminant} is given as
    \begin{equation}
        \label{eq:intro_def_of_disc}
        \disc(P) = A^{2n-2} \prod_{i<j} (\alpha_j-\alpha_i)^2 \, .
    \end{equation}
   Throughout the paper, $\xrightarrow[n\to \infty]{\bP}$ denotes convergence in probability as $n\to \infty$.
	\begin{theorem}
		\label{thm:LLN_for_L_n}
		Let $f_n$ be given by~\eqref{eq:intro_def_of_kac_polynomial} and assume that $\xi$ is a mean-zero, variance one, sub-Gaussian random variable such that $\bP(\xi=0)=0$. 
        Then there exists a constant $\sf{D}_\ast > 0$ such that
		\begin{equation*}
			\frac{1}{n}\Big(\log|\disc(f_n)| - 2n\log n\Big) \xrightarrow[n\to \infty]{\bP} -{\sf{D}_\ast}  \, . 
		\end{equation*}
	\end{theorem}  
    \noindent
    More succinctly, with high probability the discriminant of a random Kac polynomial has
    \begin{equation}
        \label{eq:intro_asymptotic_for_disc_whp}
       \qquad |\disc(f_n)| = n^{2n} e^{-n{\sf D}_{\ast}(1+ o(1))} \qquad \text{as } n\to \infty\, . 
    \end{equation} 
	We note that the constant $\sf{D}_\ast>0$ is explicit and is given by~\eqref{eq:intro_def_of_D_ast} below. Another remark is that the assumption $\bP(\xi=0)=0$ in Theorem~\ref{thm:LLN_for_L_n} is not essential, and we may drop it with additional wording; see Section~\ref{subsection:possible_double_roots} below for more details on this point.

        \subsection{Related works}
        While it is not necessary for our result to hold, a natural example to keep in mind throughout the paper is when the coefficient distribution $\xi$ is supported on the integers. In that case, the discriminant is also an integer, as it is a polynomial in the coefficients with integer coefficients, and serves as an algebraic invariant that helps studying the roots of a polynomial. 

        Our main motivation for studying the discriminant of a random Kac polynomial is the work of Bary-Soroker and Kozma~\cite{BarySoroker-Kozma}, and in particular a numerical study described in~\cite[Section~4]{BarySoroker-Kozma}. That work studies irreducibility of random polynomials with integer coefficients, and in particular establishes that $f_n$ is irreducible over $\bQ$ with high probability as $n\to\infty$, for $\xi$ which is uniformly distributed on $\{1,\ldots,240 \}$. See also~\cite{Breuillard-Varju} for a stronger result in this direction, assuming the generalized Riemann hypothesis. The irreducibility of $f_n$ is established by showing that (with high probability) the Galois group of $f_n$ contains the alternating group. In fact, in~\cite[Section~4]{BarySoroker-Kozma} Bary-Soroker and Kozma conjecture  that with high probability the Galois group of $f_n$ is the full symmetric group, which in view of the above is equivalent to $\disc(f_n)$ not being a square. In particular, based on numerical evidence, the work \cite{BarySoroker-Kozma} predicted that $\log|\Delta(f_n)|$ is asymptotically normal with linear mean and variance; our main result shows that in fact the first order term grows like
        $n \log n$ rather than 
        linearly. Motivated by this numerical study, \cite{BarySoroker-Kozma} suggests the following heuristic: the discriminant of a random polynomial is a large random integer, and hence it is unlikely for it to be a square. Theorem~\ref{thm:LLN_for_L_n} of this paper is a modest justification for the numerical observation from~\cite[Section~4]{BarySoroker-Kozma}, that the discriminant of a random polynomial is typically a very large number which concentrates on a linear scale. We also mention the recent work by Hokken~\cite{Hokken}, which computed the probability that the discriminant is a square in a related but different model of random skew-reciprocal polynomials. 

       One motivation for the previously mentioned works~\cite{BarySoroker-Kozma,Breuillard-Varju} is work on the van der Waerden conjecture.  If one considers a random monic polynomial $P$ of degree $d$ where all non-leading coefficients are integers chosen independently and uniformly at random in $[-H,H]$, van der Waerden conjectured in 1936 that the probability $P$ is reducible is $O(H^{-1})$ as $H\to\infty$, and furthermore the probability the Galois group is not the symmetric group is asymptotically equal to the probability $f$ is irreducible.  A breakthrough work of Bhargava~\cite{bhargava2025galois} proved the former statement.  Further, the work reduced the latter, stronger conjecture to showing that the probability the Galois group is the alternating group is $o(H^{-1})$.  This remains an open problem. This model for a random polynomial is often called the ``box model'', and we refer the reader to the works \cite{Bary-Soroker-BenPorath-Matei,Bary-Soroker-Goldgraber,Gotze-Zaporozhets} for more information about the discriminant of random polynomials sampled from the ``box model".

        For Gaussian random polynomials (that is, when $\xi$ is Gaussian), the behavior of $\log|\disc(f_n)|$ can be analyzed more precisely. Indeed, in~\cite{Michelen-Yakir-log-energy} the authors studied a variant of this problem on the Riemann sphere, with a different random ensemble of Gaussian polynomials. Importing the techniques from~\cite{Michelen-Yakir-log-energy}, we suspect one can show that for $\xi$ Gaussian we have that $\text{Var}\big(\log|\disc(f_n)|\big)$ grows linearly in $n$ and that the fluctuations are asymptotically normal. As this direction is quite technical and not so much related to the main motivation, we do not pursue this here.

        \subsection{Heuristic explanation for the leading term}

        At first sight, it is not at all obvious why the leading term for $\log|\Delta(f_n)|$ should be $2n\log n$, or even why it should be universal in the choice of random coefficients. Before going into technical details, we give a heuristic explanation of why this should be the case. First, recall a classical fact about the roots of random Kac polynomials: with high probability, as $n\to\infty$, the roots of $f_n$ cluster uniformly near the unit circle. In fact, this clustering occurs at the scale $n^{-1}$, meaning that there are $n\big(1-o(1)\big)$ roots present in the annulus 
        \[
        \Big\{ 1- \frac{\omega(n)}{n} \le |z| \le 1 + \frac{\omega(n)}{n} \Big\} \, ,
        \]
        as soon as $\omega(n)\to \infty$ as $n\to \infty$, see~\cite{Ibragimov-Zeitouni, Shepp-Vanderbei}. Furthermore, the angular distribution of the roots is approximately uniformly distributed, as follows from the classical Erd\H os-Tur\'an inequality~\cite{Erdos-Turan-AOM}. Denoting by $\alpha_1,\ldots,\alpha_n$ the roots of $f_n$, we have that 
        \[
        |\Delta(f_n)| = |\xi_n|^{n-2} \prod_{j=1}^{n} |f_n^\prime(\alpha_j)|
        \]
        (see the proof of Claim~\ref{claim:convenient_formula_for_disc} below). Hence, by splitting the complex plane $\bC$ into 
        \[
        {\tt I}  = \Big\{|z| < 1 - \frac{\omega(n)}{n} \Big\} \, , \qquad {\tt II} = \Big\{ 1- \frac{\omega(n)}{n} \le |z| \le 1 + \frac{\omega(n)}{n} \Big\} \, , \qquad {\tt III} = \bC\setminus({\tt I}\cup {\tt II}) \, , 
        \]
        we have that\footnote{Here and everywhere, $\alpha$ represents a symbolic root of $f_n$, and a sum or a product over $\alpha$'s represents the respective sum or product over roots.} 
        \begin{equation}
        \label{eq:intro_heuristic_for_leaing_term_in_discriminant_breakup_into_domains}
            \log|\Delta(f_n)| = \Big(\sum_{\alpha\in {\tt I}} + \sum_{\alpha\in {\tt II}} + \sum_{\alpha\in {\tt III} }\Big) \log|f_n^\prime(\alpha)| + n \log |\xi_n| + O(1) \, . 
        \end{equation}
        In the domain ${\tt I}$, it is not hard to see that $|f_n^\prime|$ grows at most polynomially, and since there are $o(n)$ roots in ${\tt I}$ we get that the corresponding sum is asymptotically negligible and does not contribute to the leading term. On the other hand, for $z\in {\tt II}$ a simple variance computation shows that typically 
        \[
        |f_n(z)| \approx \sqrt{n} \, , \qquad |f_n^\prime(z)| \approx n^{3/2} \, . 
        \]
        Since there are $n(1-o(1))$ roots in ${\tt II}$, we get that
        \[
        \sum_{\alpha \in {\tt II} } \log|f_n^\prime(\alpha)| \approx n\log(n^{3/2}) = \frac{3}{2} n \log n \, . 
        \]
        Finally, to deal with roots from the domain ${\tt III}$, we note that typically, the derivative $f_n^\prime$ grows like its leading coefficient, and we get 
        \[
        \sum_{\alpha\in {\tt III}} \log|f_n^\prime(\alpha)| + n \log|\xi_n| \approx \sum_{\alpha\in {\tt III}} \log|\xi_n n \alpha^{n-1}| + n \log|\xi_n| \approx n\sum_{\alpha\in {\tt III}} \log|\alpha| + n\log|\xi_n| \, . 
        \]
         We note that the distribution of roots in the domain ${\tt III}$ is \emph{not} universal. For instance, if $\xi \in \{-1,1\}$ then there are never roots with $|\alpha| > 2$, but if $\xi$ is Gaussian then there is a positive probability of having a root with $|\alpha| > 2$. Interestingly, the term $n\log|\xi_n|$ saves the day.
        Indeed, by applying
        Jensen's formula~\cite[p.207]{Ahlfors}
        \begin{equation}
            \label{eq:intro-Jensen_formula}
            \int_{0}^{1} \log|re^{2\pi \I \theta} - \alpha| \,  {\rm d} \theta = \begin{cases}
                \log r & |\alpha|\le r \, , \\ \log|\alpha| & |\alpha| > r \, ,
            \end{cases}
        \end{equation}
        we see that
        \[
        n\sum_{\alpha\in {\tt III}} \log|\alpha| + n\log|\xi_n| = n\int_{0}^{1}\log|f_n(e^{2\pi \I \theta})| \, {\rm d}\theta \approx n \log(n^{1/2}) = \frac{1}{2} n \log n \, .
        \]
        Plugging the above heuristic computation into~\eqref{eq:intro_heuristic_for_leaing_term_in_discriminant_breakup_into_domains} gives
        \[
        \log|\Delta(f_n)| \approx \Big(\frac{3}{2} + \frac{1}{2}\Big) n \log n = 2n\log n \, .
        \]
        
    \subsection{Possibility of double roots}
    \label{subsection:possible_double_roots}
    We briefly comment on our assumption that the coefficient distribution $\xi$ has no atom at the origin. This assumption is natural, particularly because the discriminant is not consistently defined across polynomials of differing degrees. However, it is not essential for our result.   
    Indeed, the recent work~\cite[Corollary~1.2]{Michelen-Yakir-root-separation} (see also~\cite[Theorem~1.3]{Michelen-Yakir-root-separation}) implies that 
    \[
        \lim_{n\to\infty} \bP\big(f_n \ \text{has a double root}\big) = \big(\bP(\xi_0 = 0)\big)^2 \, .
    \]
    Furthermore, it is apparent from~\eqref{eq:intro_def_of_disc} that $\disc(f_n) = 0$ if and only if $f_n$ has a double root. In the course of the proof of Theorem~\ref{thm:LLN_for_L_n}, we will in fact prove that with high probability as $n\to \infty$, we have
        \[
        |\disc(f_n)| = \begin{cases}
            0  & \text{if $f_n$ has a double root,} \\ n^{2n} e^{-n{\sf D}_{\ast}(1+ o(1))} & \text{otherwise.} 
        \end{cases}
        \]
    That is, if we do not assume anything about $\bP(\xi=0)$, then the asymptotic~\eqref{eq:intro_asymptotic_for_disc_whp} holds unless some event with (possibly) positive probability occurs. To avoid this technicality, we shall adopt the assumption $\bP(\xi_0=0) = 0$ throughout the paper. We also remark that the proof works equally well to obtain an analogous statement if one forces $\xi_0 \equiv 1$ (or equivalently $\xi_n \equiv 1$), which is another model sometimes considered, but we will not comment on this further.
    
    \subsection{Sketch of proof}
    \label{subsection:sketch_of_proof}
    For a polynomial $P(z) = A\prod_{j=1}^{n}(z-\alpha_j)$, let $P^r(z) = z^{n} P(z^{-1})$ denote its reciprocal polynomial. Then $\disc(P) = \disc(P^r)$ and for the random polynomial~\eqref{eq:intro_def_of_kac_polynomial}, $f_n$ and $f_n^r$ have the same law. The first key step towards the proof of Theorem~\ref{thm:LLN_for_L_n} is a convenient representation of the discriminant, which more visibly ``sees'' this distributional symmetry. Recall that the \emph{Mahler measure} of a polynomial $P$ is given by
	\begin{equation}
		\label{eq:def_of_Mahler_measure}
		M(P) = A \prod_{j=1}^n \max\{1,|\alpha_j|\}  = \exp\left(\int_{0}^{1} \log|P(e^{ 2\pi \I \theta})| \, \dd \theta\right)
	\end{equation}
	where the last equality is Jensen's formula~\eqref{eq:intro-Jensen_formula}.  Although it is called a ``measure", the Mahler measure is in fact a multiplicative height function on polynomials, which is another well-studied algebraic invariant; see the book~\cite{McKee-Smyth} for more background and references. Its relevance to our study of the discriminant is due to the following simple claim, which is a ``symmetrized" version of~\eqref{eq:intro_heuristic_for_leaing_term_in_discriminant_breakup_into_domains}. 
	\begin{claim}
	\label{claim:convenient_formula_for_disc}
		 For any polynomial $P$ we have
		 \[
		 |\disc(P)| = \Big(\prod_{|\alpha|\le 1}|P^\prime(\alpha)| \Big) \cdot \Big(\prod_{|\alpha|>1} \left|\frac{P^\prime(\alpha)}{\alpha^{n-2}}\right| \Big) \cdot M(P)^{n-2}  \, ,
		 \]
         where the product is over all roots of $P$.
	\end{claim}
	\begin{proof}
		For $P(z) = A\prod_{j=1}^n (z-\alpha_j)$, we have for any root $\alpha_j$ that
		\[
		P^\prime(\alpha_j) = A\prod_{i\not= j} (\alpha_j-\alpha_i) \, .
		\]
		Using the definition of the discriminant we see that
		\begin{equation*}
			|\disc(P)| = |A|^{2n-2} \prod_{j=1}^n \Big(\prod_{i\not=j} |\alpha_i-\alpha_j|\Big) = |A|^{n-2} \prod_{j=1}^{n} |P^\prime(\alpha_j)|
		\end{equation*}
		This is nothing but the representation of the discriminant as the resultant of the polynomials and its derivative. The claim now follows immediately from~\eqref{eq:def_of_Mahler_measure}.
	\end{proof}
    
    The main advantage of this representation for the discriminant is that for the random polynomial $f_n$ the variables
    \begin{equation}
        \label{eq:intro_terms_with_same_law}
        \prod_{|\alpha|<1} |f_n^{\prime}(\alpha)|   \qquad \text{and} \qquad \prod_{|\alpha|>1} \Big|\frac{f_n^\prime(\alpha)}{\alpha^{n-2}}\Big| 
    \end{equation}
        have the same law. Though perhaps not immediately obvious, \eqref{eq:intro_terms_with_same_law} is a consequence of the reciprocal symmetry, and is proved in Claim~\ref{claim:inside-outside-equidist} below. 
    Therefore, with Claim~\ref{claim:convenient_formula_for_disc} we can rewrite $\log|\disc(f_n)|$ as a sum of three random terms (see Lemma~\ref{lemma:expression_log_discriminant} below), two of which are equal in law, and our goal is the prove concentration for each term on the linear scale, as $n\to \infty$. First, we state the concentration result for the Mahler measure term:
    \begin{proposition}
		\label{prop:LLN_for_log_mahler_measure}
		Let $f_n$ be the random polynomial~\eqref{eq:intro_def_of_kac_polynomial}. Then 		
		\[
		\int_0^1\log \left|\frac{f_n(e^{2\pi \I \theta})}{\sqrt{n}} \right|\,{\rm d}\theta \xrightarrow[n\to \infty]{\bP} -\frac{\gamma}{2} \, ,
		\]
        where $\gamma\approx 0.5772$ is Euler's constant. 
	\end{proposition} 
    Stated more explicitly, Proposition~\ref{prop:LLN_for_log_mahler_measure} shows concentration for the Mahler measure of random polynomials, namely that
    \begin{equation*}
        \label{eq:intro-concentration_for_mahler_measure}
        M(f_n) = \sqrt{n}  \, e^{-\gamma/2}\big(1+o(1)\big) \, ,
    \end{equation*}
    with high probability as $n\to \infty$. 
    We remark that a special case of Proposition~\ref{prop:LLN_for_log_mahler_measure}, when $\xi$ takes the values $\{-1,+1\}$ with equal probability, was already proved as an auxiliary result in~\cite[Lemma~1.3]{Yakir-Studia}. In Section~\ref{sec:concentration_of_the_Mahler_measure} we will provide an alternative simple proof of Proposition~\ref{prop:LLN_for_log_mahler_measure}, which uses a result of Cook-Nguyen~\cite{Cook-Nguyen} and allows for more general distributions of coefficients.

    The bulk of the technical work in this paper is to prove the corresponding concentration result for the terms~\eqref{eq:intro_terms_with_same_law}. For $t\ge 0$ we set
    \begin{equation}
        \label{eq:Phi-def}
        \Phi(t) = \Big(\frac{1}{t^2} - \frac{1}{\sinh^2(t)}\Big) \cdot \log \bigg( \frac{\big(1+2t^2-\cosh(2t)\big)\cdot \big( 1-\coth(t)\big)}{2t^3} \bigg)  \, .
        \end{equation}
        Here, $\sinh,\cosh$ and $\coth$ are the hyperbolic sine, cosine and cotangent, respectively. Obviously, $\Phi$ is a continuous function  and a simple computation shows
        \[
        \lim_{t\to 0^+} \Phi(t) = -\frac{\log 3}{3} \, , \qquad \text{and} \qquad \lim_{t\to\infty} \frac{t^2}{\log t} \, \Phi(t) = -3 \, . 
        \]
    Therefore, $\Phi$ is an integrable function on $\bR_{\ge 0}$, and we can set
    \begin{equation}
        \label{eq:def_of_c_ast}
        {\sf c}_\ast = 1-\gamma + \int_{0}^\infty \Phi(t) \, {\rm d}t \, .
    \end{equation}
    
    \begin{proposition}\label{prop:sum-concentrated}
        Let $f_n$ be the random polynomial \eqref{eq:intro_def_of_kac_polynomial} and assume that $\bP(\xi = 0 )= 0$.  Then $$ \frac{1}{n} \sum_{|\alpha| < 1} \log \left|\frac{f_n'(\alpha)}{n^{3/2}} \right| \xrightarrow[n\to \infty]{\bP} {\sf c}_\ast \, ,$$
        where ${\sf c}_\ast$ is given by~\eqref{eq:def_of_c_ast}.
    \end{proposition}
    We note that for $z$ close to the unit circle, the standard deviation of $f_n^\prime(z)$ is of order $n^{3/2}$, and hence the normalization appearing in Proposition~\ref{prop:sum-concentrated} is natural. In view of the above, we can read out an explicit formula for the limiting constant ${\sf D}_\ast>0$ appearing in Theorem~\ref{thm:LLN_for_L_n}. It is given as 
    \begin{equation}
	    \label{eq:intro_def_of_D_ast}
        {\sf D}_\ast  = \frac{\gamma}{2} - 2(1-\gamma) - 2\int_{0}^\infty \Phi(t) \, {\rm d}t \, , 
	\end{equation}
    where $\gamma\approx 0.5772$ is Euler's constant and $\Phi$ is given by~\eqref{eq:Phi-def}. A numerical approximation of the integral appearing above gives that ${\sf D}_\ast \approx 5.92947$.
    
     To handle the sum from Proposition~\ref{prop:sum-concentrated} we recall that  with high probability as $n\to \infty$, the roots of $f_n$ cluster uniformly around the unit circle at scale $n^{-1}$.
    Therefore, it is reasonable to guess that 
    \begin{equation}
        \label{eq:intro_sum_in_disk_is_negligable}
        \frac{1}{n} \sum_{|\alpha| < 1 - \frac{\log^3n}{n}} \log \left|\frac{f_n'(\alpha)}{n^{3/2}} \right| \xrightarrow[n\to \infty]{\bP} 0\, ,
    \end{equation}
    which is indeed the case, see Lemma~\ref{lemma:sum-inside-unit-disk} and Lemma~\ref{lemma:sum_in_the_annulus_weak} below. In view of~\eqref{eq:intro_sum_in_disk_is_negligable}, Proposition~\ref{prop:sum-concentrated} boils down to showing that 
    $$ \frac{1}{n} \sum_{1- \log^3 n / n < |\alpha| < 1} \log \left| \frac{f_n'(\alpha)}{n^{3/2}}\right| \xrightarrow[n\to \infty]{\bP} {\sf c}_\ast \, ,$$
    which we prove in Section~\ref{sec:lln_for_sum_in_annulus}. 
 
     A key step in the analysis is to “remove the singularity” caused by the derivative being small at the roots. For instance, if the random polynomial $f_n$ satisfies
    \[
    \bP\big(\exists \alpha \, ; \, f_n(\alpha) = f_n^\prime(\alpha) = 0 \big) >0 \, ,
    \]
    as is often the case for discrete coefficients distribution $\xi$, then the sum in Proposition~\ref{prop:sum-concentrated} has infinite mean. To establish concentration, we therefore need to isolate the contribution of roots on which  $f_n^\prime$ is atypically small, and show that it is asymptotically negligible. This is done by importing some technical results from our work~\cite{Michelen-Yakir-root-separation}, see Section~\ref{sec:negligible_contributions} for more details. After that, we have reduced the proof of Proposition~\ref{prop:sum-concentrated} to proving that 
    \begin{equation}
        \label{eq:intro_sum_concentrated_after_reductions}
        \frac{1}{n} \sum_{1- \log^3 n / n < |\alpha| < 1} \log \Big| \frac{f_n'(\alpha)}{n^{3/2}}\Big| \cdot 
        \one\Big\{\frac{|f_n^\prime(\alpha)|}{n^{3/2}}\ge \frac1{\polylog(n)}\Big\} \xrightarrow[n\to \infty]{\bP} {\sf c}_\ast \, .
    \end{equation}
    To study the sum in~\eqref{eq:intro_sum_concentrated_after_reductions}, we apply a net argument. More explicitly, we construct a net $\mathcal{N}$ of mesh size $n^{-1-\beta}$ (for some small, fixed $\beta > 0$) in the annulus $\{1 - \log^3 n / n \le |z| \le 1\}$. Each net point $z \in \mathcal{N}$ is associated with a polar rectangle $\mathcal{R}_z$ of area $n^{-1 - 2\beta}$. We then approximate the event that a root lies in $\mathcal{R}_z$ by an event involving the linear approximation of $f_n$ at $z$. This approximation enables us to transfer the analysis from the sum~\eqref{eq:intro_sum_concentrated_after_reductions} over roots to a similar sum over the net $\cN$. A key feature of the net $\mathcal{N}$ is its relatively coarse width. For $z\in \cN$, the probability that $\mathcal{R}_z$ contains a root is approximately $n^{-2\beta}$. If we take $\beta$ to be sufficiently small, this allows us to use a standard application of the Berry–Esseen theorem to compare the probabilities of these events with those arising from the `Gaussian version' of $f_n$, i.e. when the coefficients are Gaussian random variables. Through this Gaussian comparison, we show that the sum in~\eqref{eq:intro_sum_concentrated_after_reductions} is close in $L^1(\bP)$ to the corresponding sum over $\cN$, and that the variance of the latter tends to zero as $n \to \infty$. 
    
    The sketch described in the paragraph above establishes the concentration~\eqref{eq:intro_sum_concentrated_after_reductions}, and it remains to identify the actual limit by computing the limit of expectations. Indeed, the constant ${\sf c}_\ast$ arises from an explicit computation of the expected sum in the case where $f_n$ has Gaussian coefficients. In that setting, a Kac–Rice type formula yields an exact expression for the expected sum over roots. After that, a straightforward asymptotic analysis of the formula gives the limiting expectation ${\sf c}_\ast$, which (by the Gaussian comparison) also governs the limiting expectation in the general (non-Gaussian) case.
    
    \subsection*{Organization of the paper}
    To ease on the readability, we outline the plan for the rest of the paper: 
    \begin{itemize}
        \item[$\circ$] In Section~\ref{sec:breakdown_of_the_proof} we collect some preliminary results and a more elaborate breakdown of the proof of our main result. In particular, the proof of Theorem~\ref{thm:LLN_for_L_n} assuming Proposition~\ref{prop:LLN_for_log_mahler_measure} and Proposition~\ref{prop:sum-concentrated} is given in this section. \vspace{4pt}

        \item[$\circ$] In Section~\ref{sec:negligible_contributions} we justify~\eqref{eq:intro_sum_in_disk_is_negligable}. That is, we show that the contribution of roots away from the unit circle to the sum appearing in Proposition~\ref{prop:sum-concentrated} is negligible. We will also use this section to show that the contribution from roots on which the derivative is  small is also asymptotically negligible. \vspace{4pt}

        \item[$\circ$] In Section~\ref{sec:lln_for_sum_in_annulus} we complete the proof of Proposition~\ref{prop:sum-concentrated}, by reducing to a net argument and showing that the sum over the net points is concentrated. \vspace{4pt}

        \item[$\circ$] In Section~\ref{sec:computing_the_mean} we show that the mean for the sum over the net points converges, and by reducing to the case of Gaussian polynomials we compute its limiting value via the Kac-Rice formula. \vspace{4pt}

        \item[$\circ$] In Section~\ref{sec:concentration_of_the_Mahler_measure} we give the proof of Proposition~\ref{prop:LLN_for_log_mahler_measure}, which shows concentration of the Mahler measure of random polynomials. \vspace{4pt}

        \item[$\circ$] Finally, in Section~\ref{sec:gaussian_comparison}, we prove some standard Gaussian comparison results (based on the Berry-Esseen theorem) that we need to use along the way.
    \end{itemize}

    \subsection*{Notation}

    We list here some notation which we use across the paper:
    \begin{itemize}
        \item $f_n$ the random Kac polynomial~\eqref{eq:intro_def_of_kac_polynomial} of degree $n$; $\xi$ its coefficients distribution; \vspace{4pt}
        \item $\bR,\bC$ the real line and the complex plane; $m$ the Lebesgue measure in $\bC^d\simeq \bR^{2d}$ (the dimension should be clear from the context); \vspace{4pt}
        \item $\bD,\bS^1$ the unit disk and the unit circle; \vspace{4pt}
        \item $\alpha$ is always a symbolic root of some underlying polynomial $P$. For instance
        \[
        \sum_{\alpha} F(P,\alpha) = \sum_{\alpha: P(\alpha)=0} F(P,\alpha) \, .
        \]
        \item $\disc(P)$ the discrimiant~\eqref{eq:intro_def_of_disc} of a polynomial $P$; $M(P)$ the Mahler measure~\eqref{eq:def_of_Mahler_measure};    \vspace{4pt}
        \item We denote by $X_n\xrightarrow[n\to\infty]{\bP}c$ for convergence in probability. That is, for all $\eps>0$
        \[
        \lim_{n\to\infty} \bP\big(|X_n-c|\ge \eps\big) = 0\, .
        \]
        \item As usual, we write $\log_+(x) = \max\{0,\log(x)\}$ and $\log_-(x) = \log_+(x)-\log(x)$, and so
        \[
        |\log(x)| = \log_+(x) + \log_-(x) \, .
        \]         
        \end{itemize}
        We will use freely the Landau notations $O(\cdot),o(\cdot),\Omega(\cdot),\Theta(\cdot)$ to denote asymptotic inequalities up to constants which do not depend on $n$. We will also write $X\lesssim Y$ for $X=O(Y)$. Finally, it will be convenient for us to use $\widetilde O(\cdot), \widetilde \Omega(\cdot)$ to denote inequalities up to non-asymptotic constants and $\polylog(n)$ terms. For example, if $X = O(Y \log^{2025}(n))$ then also $X=\widetilde{O}(Y)$.
    
    \subsection*{Acknowledgments}
    M.M.\ is supported in part by NSF CAREER grant DMS-2336788 as well as DMS-2246624. O.Y.\ is supported in part by NSF postdoctoral fellowship DMS-2401136.

    \section{Breakdown of the proof of Theorem~\ref{thm:LLN_for_L_n}}
    \label{sec:breakdown_of_the_proof}
    
    We begin this section by collecting a few preliminaries, and then see how we can deduce Theorem \ref{thm:LLN_for_L_n} from Proposition~\ref{prop:LLN_for_log_mahler_measure} and Proposition~\ref{prop:sum-concentrated} stated in Section~\ref{subsection:sketch_of_proof}.
    
    \subsection{Preliminaries}
    As we already mentioned in the introduction, the starting point of the proof is a convenient representation for the logarithm of the discriminant, which follows from Claim~\ref{claim:convenient_formula_for_disc} in the introduction. This representation will allow us the ``see" the concentration on scale $n$, as $n\to \infty$. 
     \begin{lemma}\label{lemma:expression_log_discriminant}
        Let $P$ be a polynomial of degree $n$ with no double roots and no root on $\bS^1$. Then \begin{align*}
            \log |\disc(P)| &=  \sum_{|\alpha| < 1} \log \left|P'(\alpha)\right| + \sum_{|\alpha| > 1} \log \left|\frac{P'(\alpha)}{\alpha^{n-2}}\right| + (n-2) \int_0^1 \log |P(e^{2\pi \I \theta})| \,{\rm d}\theta \\ 
            &= \sum_{|\alpha| < 1} \log \left|\frac{P'(\alpha)}{n^{3/2}}\right| + \sum_{|\alpha| > 1} \log \left|\frac{P'(\alpha)}{\alpha^{n-2} n^{3/2}}\right| + (n-2) \int_0^1 \log \left|\frac{P(e^{2\pi \I \theta})}{\sqrt{n}}\right| \,{\rm d}\theta  + (2n-1)\log n \, ,
        \end{align*}
        where the sums are over roots of $P$.
    \end{lemma}
    \begin{proof}
        If $P$ has no double roots, then $P'(\alpha) \neq 0$ for all roots $\alpha$.  We may then take the logarithm of both sides of the equality from Claim~\ref{claim:convenient_formula_for_disc}, which in view of~\eqref{eq:def_of_Mahler_measure} proves the first equality. Adding and subtracting $(2n - 1) \log n$ and distributing $\frac{3}{2}\log n$ to each of the $n$ roots, as well as $(n-2)\log\sqrt{n}$ to the Mahler measure term completes the proof.
    \end{proof}
    We note that Proposition~\ref{prop:LLN_for_log_mahler_measure} shows concentration for the integral (i.e. the Mahler measure term) on the right-hand side of Lemma~\ref{lemma:expression_log_discriminant}, while Proposition~\ref{prop:sum-concentrated} shows concentration for the first sum appearing in this expression. Next on the agenda, we want to show that for random polynomials, both sums in fact have the same law. 
    \begin{claim}\label{claim:inside-outside-equidist}
    Let $f_n$ be the random polynomial~\eqref{eq:intro_def_of_kac_polynomial}, then the random variables
		\[
		\sum_{|\alpha|<1} \log\left|\frac{f_n^\prime (\alpha)}{n^{3/2}}\right| \qquad \text{and} \qquad \sum_{|\alpha|>1} \log\left|\frac{f_n^\prime(\alpha)}{\alpha^{n-2} n^{3/2}}\right|
		\]
		have the same distribution.
	\end{claim}
	\begin{proof}
		Define $g_n(z) = z^{n} f_n(z^{-1}) = \xi_0 z^n + \xi_1 z^{n-1} + \ldots+ \xi_n$. Then $g_n\stackrel{d}{=}f_n$, and every root of $f_n$ corresponds to a reciprocal root of $g_n$ (and vice-versa). Denoting by $\beta$ a symbolic root of $g_n$, we have
		\[
		g_n^{\prime}(\beta) = n\beta^{n-1}f_n(\beta^{-1}) - \beta^{n-2} f_n^\prime(\beta^{-1}) = -\beta^{n-2} f_n^{\prime}(\beta^{-1}) \, . 
		\]
		We get that
		\begin{align*}
			\sum_{|\alpha|<1} \log\left|\frac{f_n^\prime (\alpha)}{n^{3/2}}\right| \stackrel{d}{=} \sum_{|\beta|<1} \log\left|\frac{g_n^\prime (\beta)}{n^{3/2}}\right|  = \sum_{|\beta|<1} \log \left|\beta^{n-2} \frac{ f_n^\prime(\beta^{-1})}{n^{3/2}}\right| = \sum_{|\alpha|>1} \log\left|\frac{f_n^\prime(\alpha)}{\alpha^{n-2} n^{3/2}}\right| \, , 
		\end{align*}
		as desired.
	\end{proof}

    \subsection{Proof of the main result}
    After collecting the necessary preliminaries, we now show how the proof of Theorem~\ref{thm:LLN_for_L_n} follows from Proposition~\ref{prop:LLN_for_log_mahler_measure} and Proposition~\ref{prop:sum-concentrated}. To apply Lemma~\ref{lemma:expression_log_discriminant} with $P = f_n$, we need to know that with high probability there are no roots of $f_n$ on the unit circle itself. While proving this fact is not very difficult, we state  a quantitative statement, deduced from a much stronger statement due to Cook and Nguyen~\cite{Cook-Nguyen}, which we will also use when analyzing the Mahler measure of a random polynomial. 
    
    \begin{claim}
        \label{claim:no_root_on_unit_circle}
        Let $f_n$ be the random polynomial~\eqref{eq:intro_def_of_kac_polynomial}, then
        \[
        \lim_{n\to\infty} \bP\Big( \min_{z\in \bS^1} |f_n(z)| \ge \frac{1}{n} \Big) = 1 \, .      
        \]
        In particular, we also have
        \[
        \lim_{n\to \infty} \bP\Big( \, \exists \alpha\in \bS^1 \, : \, f_n(\alpha) = 0 \Big) = 0 \, .
        \]
    \end{claim}
    \begin{proof}
        The claim follows from a universality result on the minimum modulus on the unit circle of random polynomials~\cite[Theorem~1.2]{Cook-Nguyen} (building on the same result for Gaussian polynomials~\cite{Yakir-Zeitouni}). It states that
        \[
        \sqrt{n} \min_{z\in \bS^1} |f_n(z)|
        \]
        converges in law as $n\to \infty$ to a non-trivial exponential random variable. In particular
        \[
        \bP\Big( \, \exists \alpha\in \bS^1 \, : \, f_n(\alpha) = 0 \Big) \le  \bP\Big( \sqrt{n} \min_{z\in \bS^1} |f_n(z)| \le \frac{1}{\sqrt{n}} \, \Big)\xrightarrow{n\to\infty} 0 \, ,
        \]
        as claimed.
    \end{proof}
    We are ready to prove the main result of the paper. 
    \begin{proof}[Proof of Theorem~\ref{thm:LLN_for_L_n}]
        By~\cite[Corollary~1.2]{Michelen-Yakir-root-separation} we have that
        \[
        \lim_{n\to \infty} \bP\Big( \, f_n \ \text{has a double root} \, \Big) = 0 \, .
        \]
        Hence, in view of Claim~\ref{claim:no_root_on_unit_circle}, we have that $\bP(\mathcal{T}) = 1 - o(1)$ as $n\to\infty$, where
        \[
        \mathcal{T} = \Big\{ \text{$f_n$ has no double roots and no root on $\bS^1$}\Big\} \, .
        \]
        On the event $\mathcal{T}$, we can apply Lemma~\ref{lemma:expression_log_discriminant} and get that
        \begin{multline}
        \label{eq:proof_of_main_result_after_application_of_convinient_expression}
            \frac{1}{n}\Big(\log|\disc(f_n)| - 2n\log n\Big) \\ = \frac{1}{n}\sum_{|\alpha| < 1} \log \left|\frac{f_n'(\alpha)}{n^{3/2}}\right| + \frac{1}{n}\sum_{|\alpha| > 1} \log \left|\frac{f_n'(\alpha)}{\alpha^{n-2} n^{3/2}}\right| + \frac{n-2}{n} \int_0^1 \log \left|\frac{f_n(e^{2\pi \I \theta})}{\sqrt{n}}\right| \,{\rm d}\theta  + o(1) \, .
        \end{multline}
        By Proposition~\ref{prop:LLN_for_log_mahler_measure} we have that
        \[
        \frac{n-2}{n} \int_0^1 \log \left|\frac{f_n(e^{2\pi \I \theta})}{\sqrt{n}}\right| \,{\rm d}\theta \xrightarrow[n\to\infty]{\bP} -\frac{\gamma}{2}\, , 
        \]
        whereas by Proposition~\ref{prop:sum-concentrated} we have
        \[
        \frac{1}{n}\sum_{|\alpha| < 1} \log \left|\frac{f_n'(\alpha)}{n^{3/2}}\right| \xrightarrow[n\to\infty]{\bP} {\sf c}_\ast \, .
        \]
        Finally, by Claim~\ref{claim:inside-outside-equidist} both sums on the right-hand side of~\eqref{eq:proof_of_main_result_after_application_of_convinient_expression} have the same law, and another application of Proposition~\ref{prop:sum-concentrated} shows that 
        \[
        \frac{1}{n}\sum_{|\alpha| > 1} \log \left|\frac{f_n'(\alpha)}{\alpha^{n-2} n^{3/2}}\right| \xrightarrow[n\to\infty]{\bP} {\sf c}_\ast \, .
        \]
        Altogether, as $\bP(\mathcal{T}^c) \xrightarrow{n\to \infty} 0$, relation~\eqref{eq:proof_of_main_result_after_application_of_convinient_expression} yields that
        \[
        \frac{1}{n}\Big(\log|\disc(f_n)| - 2n\log n\Big) \xrightarrow[n\to \infty]{\bP} -\frac{\gamma}{2}+2 \, {\sf c}_\ast \, . 
        \]
        It remains to note that
        \[-\frac{\gamma}{2}+2 \, {\sf c}_\ast \stackrel{\eqref{eq:def_of_c_ast}}{=} -\frac{\gamma}{2} + 2(1-\gamma)+2\int_0^\infty \Phi(t) \, {\rm d}t \stackrel{\eqref{eq:intro_def_of_D_ast}}{=} - {\sf D}_\ast \, ,
        \]
        which is what we wanted to show. 
    \end{proof}

    \subsection{Control on the maximum}
    We conclude this section by collecting some basic analytic facts that will be helpful in the proofs of Proposition~\ref{prop:LLN_for_log_mahler_measure} and Proposition~\ref{prop:sum-concentrated}. We start by showing how our sub-Gaussian assumption on the random coefficients allows us to control the maximum of the random polynomial $f_n$ (and its first two derivatives) via the Salem-Zygmund inequality (see~\cite[Chapter~6]{Kahane}). Crucially, the maximal values are $\polylog(n)$ away from the typical values, which is the context of our next lemma. 
   
    \begin{lemma} \label{lemma:Salem-Zygmund}
    For the event
     \begin{equation} \label{eq:G-def}
    \mathcal{G} = \bigcap_{\ell=0}^{2} \Big\{ \max_{\theta\in[0,2\pi]} \big|f_n^{(\ell)}\big((1+1/n) e^{\I \theta}\big)\big| \le n^{\ell + 1/2} \log n\Big\}  \, ,
    \end{equation}
    there exists $c>0$ so that $\P(\cG) \ge 1 -  e^{-c\log^2 n}$.
    \end{lemma}
    \begin{proof}
        For $\ell\in\{0,1,2\}$ we have 
        \[
        \sum_{k=0}^{n} k^{2\ell} \big(1+1/n\big)^{2k} \lesssim \sum_{k=0}^{n} k^{2\ell} \lesssim n^{2\ell + 1} \, . 
        \]
        Therefore, for some $c>0$ small enough we have 
        \begin{multline*}
            \bP\Big(\max_{\theta\in[0,2\pi]} \big|f_n^{(\ell)}\big((1+1/n) e^{\I \theta}\big)\big| \ge n^{\ell + 1/2} \log n\Big) \\ \le \bP\bigg(\max_{\theta\in[0,2\pi]} \big|f_n^{(\ell)}\big((1+1/n) e^{\I \theta}\big)\big| \ge c \Big(\sum_{k=0}^{n} k^{2\ell} \big(1+1/n\big)^{2k}\Big)^{1/2} \log n\bigg) \le \exp\Big(-c\log^2n \Big) \, .
        \end{multline*}
        The second inequality in the above display is the Salem-Zygmund inequality \cite[Chapter~6, Theorem~1]{Kahane}. By union bounding the above over $\ell \in \{0,1,2 \}$ we get what we want.  
    \end{proof}
    We will also make use of a Jensen-type bound on the number of roots of an analytic function, in terms of its growth.
    \begin{claim}\label{claim:Jensen_bound_on_number_of_roots}
    Suppose $f$ is an analytic function in $\bD$, continuous up to the boundary $\bS^1$, and such that $f(0)\not=0$. Then for every $r \in (0,1)$ we have
    \begin{equation*}
        \#\big\{|\alpha|\le r \, : \, f(\alpha) = 0\big\} \le \frac{1}{1-r}\log\Big(\frac{M_f}{|f(0)|}\Big) \, ,
    \end{equation*}
    where $M_f = \displaystyle \max_{\theta \in [0,2\pi]}|f(e^{\I\theta})|$.
    \end{claim}
    \begin{proof}
        By Jensen's formula we have 
        \[
        \int_{0}^{1} \log|f(e^{2\pi\I\theta})|\, {\rm d}\theta - \log|f(0)| = \sum_{|\alpha|<1} \log\Big(\frac{1}{|\alpha|}\Big) \, .
        \]
        Clearly
        $
        \displaystyle \int_{0}^{1} \log|f(e^{2\pi\I\theta})|\, {\rm d}\theta \le \log M_f \, , 
        $
        and furthermore
        \[
        \sum_{|\alpha|<1} \log\Big(\frac{1}{|\alpha|}\Big) \ge \#\big\{|\alpha|\le r \, : \, f(\alpha) = 0\big\} \log\Big(\frac{1}{r}\Big) \, .
        \]
        Altogether, we get the inequality
        \[
        \#\big\{|\alpha|\le r \, : \, f(\alpha) = 0\big\} \le \frac{1}{\log(1/r)}\log\Big(\frac{M_f}{|f(0)|}\Big) \, ,
        \]
        and by bounding $\log(1/r) \geq 1 - r$ for $r\in(0,1)$ we get the claim.
    \end{proof}

    \section{Negligible contributions away from the annulus}
    \label{sec:negligible_contributions}
    \noindent
    The goal of this section is to prove~\eqref{eq:intro_sum_in_disk_is_negligable} from the introduction. That is, to show that roots which are sufficiently far away from the unit circle do not contribute to the sum in Proposition~\ref{prop:sum-concentrated}. We will show~\eqref{eq:intro_sum_in_disk_is_negligable} in two steps. We first handle the sum of roots with $|\alpha|\le 1-\delta$ for some fixed but small $\delta>0$, for which we lean on our recent result~\cite[Theorem~1.3]{Michelen-Yakir-root-separation} about double zeros of infinite random power series. After that, we turn to handle roots in the annulus $\{ 1-\delta \le |\alpha| \le 1-\log^3 n /n\}$, where we again lean on our work~\cite{Michelen-Yakir-root-separation} to show it is uncommon to have a root in this annulus on which the derivative is small.
    
    \subsection{Roots from inside the unit disk}
     In this section we deal with roots with $|\alpha| \le 1-\delta$, for some small but fixed $\delta>0$.  
       \begin{lemma}\label{lemma:sum-inside-unit-disk}
        For each fixed $\delta > 0 $ we have $$\frac{1}{n} \sum_{|\alpha| < 1- \delta} \log \left| \frac{f_n'(\alpha)}{n^{3/2}}\right| \xrightarrow[n\to\infty]{\bP} 0 \, .$$
    \end{lemma}
    \noindent
    For a sequence of i.i.d.\ random variables $\xi_k$, all with common sub-Gaussian distribution $\xi$, we set
    \begin{equation}
    \label{eq:def_of_infinite_power_series}
	    f_\infty(z) = \sum_{k = 0}^\infty \xi_k z^k\, .
	\end{equation}
    By the Borel-Cantelli lemmas, it is not hard to see that almost surely $f_\infty$ defines a random analytic function in $\bD$. As this analytic function is also the distributional limit of $f_n$ in the disk $\{|z|\le 1-\delta\}$ as $n\to \infty$, the following theorem from our work~\cite{Michelen-Yakir-root-separation} becomes relevant. 
    \begin{theorem}[{\cite[Theorem~1.3]{Michelen-Yakir-root-separation}}]
    \label{thm:almost_sure_no_double_roots_for_random_power_series}
		Let $f_\infty$ be given by~\eqref{eq:def_of_infinite_power_series}, and assume that $\bP(\xi=0) = 0$. Then almost surely $f_\infty$ does not have a double zero in $\bD$. 
	\end{theorem}  
    \begin{proof}[Proof of Lemma~\ref{lemma:sum-inside-unit-disk}]
        We will in fact prove a slightly stronger statement: let $\{\xi_k\}_k$ be a sequence of i.i.d.\ random variables, and let $f_n$ and $f_\infty$ be the corresponding Kac polynomial~\eqref{eq:intro_def_of_kac_polynomial} and infinite Taylor series~\eqref{eq:def_of_infinite_power_series}, respectively. As a consequence of the convergence of~\eqref{eq:def_of_infinite_power_series}, we have that almost surely $f_n$ converges to $f_\infty$ uniformly on compact subsets of $\bD$, and in particular in $\{|z|\le 1-\delta\}$. We will show that
        \begin{equation}
        \label{eq:sum_inside_the_disk_negligible_almost_surely}
            \frac{1}{n} \sum_{|\alpha| < 1- \delta} \log \left| \frac{f_n'(\alpha)}{n^{3/2}}\right| \xrightarrow[n\to\infty]{\text{a.s.}} 0 \, ,
        \end{equation}
        with respect to this coupling with $f_\infty$. As almost-sure convergence implies convergence in probability, Lemma~\ref{lemma:sum-inside-unit-disk} will follow once~\eqref{eq:sum_inside_the_disk_negligible_almost_surely} is established. Indeed,  by Hurwitz's theorem the number of roots of $f_n$ in $\{|z|\le 1-\delta\}$ is eventually at most the number of roots of $f_\infty$ in $\{|z| \leq 1 - \delta/2$\} which is also almost surely finite by the identity theorem for analytic functions. Hence
        \[
        \frac{\log n}{n} \,  \# \big\{ |\alpha|\le 1-\delta \, :\, f_n(\alpha) =0\big\} \xrightarrow[n\to\infty]{\text{a.s.}} 0 \, .
        \]
     It remains to deal with a potential singularity of $\log|f_n^\prime|$ at a root. We already asserted that almost surely $f_\infty$ has only finitely many zeros in $\{|z|\le 1-\delta\}$. By Theorem~\ref{thm:almost_sure_no_double_roots_for_random_power_series}, all of these zeros have multiplicity one, and hence
     \[
     \sum_{\substack{|\alpha| < 1- \delta \\ f_\infty(\alpha)=0}} \log | f_\infty^\prime(\alpha)|
     \]
     is almost surely finite. By Hurwitz's theorem, almost surely for $n$ large enough $f_n$ has no double roots, and we get that
     \[
     \sum_{\substack{|\alpha| < 1- \delta \\ f_n(\alpha)=0}} \log | f_n^\prime(\alpha)| \xrightarrow[n\to\infty]{\text{a.s.}}  \sum_{\substack{|\alpha| < 1- \delta \\ f_\infty(\alpha)=0}} \log | f_\infty^\prime(\alpha)|\, , 
     \]
     and in particular
     \[
     \frac{1}{n}\sum_{\substack{|\alpha| < 1- \delta \\ f_n(\alpha)=0}} \log | f_n^\prime(\alpha)| \xrightarrow[n\to\infty]{\text{a.s.}} 0 \, .
     \]
     This proves~\eqref{eq:sum_inside_the_disk_negligible_almost_surely}, and hence also the lemma.  
    \end{proof}

    \subsection{Roots on which the derivative is small I}
    The next thing on our agenda is to get rid of intermediate roots, i.e.\ those roots in the annulus $\{ 1-\delta \le |z| \le 1-\log^3 n / n \}$ for $\delta>0$ small and fixed. 
    \begin{lemma}
        \label{lemma:sum_in_the_annulus_weak}
        There exists $\beta>0$ and $C>0$ so that for all $\delta>0$ we have
        \[
        \limsup_{n\to \infty} \, \bP\Bigg( \ \Bigg| \sum_{1-\delta\le |\alpha|\le 1-\log^3 n /n} \log \bigg|\frac{f_n^\prime(\alpha)}{n^{3/2}}\bigg|\, \Bigg| \ge \frac{n}{\sqrt{\log n}} \, \Bigg) \le C \delta^\beta \, .
        \]
    \end{lemma}
    To prove Lemma~\ref{lemma:sum_in_the_annulus_weak}, we will need the following claim, which is implicit in~\cite{Michelen-Yakir-root-separation}. 
    \begin{claim}\label{claim:no-small-deriv-weak}
        There exists $\beta>0$ small enough and $C>0$ so that for all $\delta>0$ we have 
        \[
        \bP\Big(\exists \alpha \in \Big\{1-\delta \le |z| \le 1 - \frac{\log^3(n)}{n} \Big\}  \, : \, f_n(\alpha)=0 \quad  \text{and} \quad |f_n^\prime(\alpha)|\le 1 \Big) \le C\delta^\beta
        \]
        for all $n$ large enough. 
    \end{claim}
    \begin{proof}
        Let $N = \lfloor \log_2\big(\delta n - \log^3(n) + 1\big) \rfloor + 1$. For $0\le k\le N$ we set
        \[
        \rho_{k} = 1-\frac{\log^3(n) + 2^k - 1 }{n} \, , 
        \]
        and so, $\rho_0 = 1-\log^3(n)/n$ and $1-2\delta\le \rho_N\le 1-\delta$. By~\cite[Lemma~4.3]{Michelen-Yakir-root-separation}, there exists $\beta>0$ small so that
        \begin{equation}
        \label{eq:bound_on_probability_root_with_small_derivative_per_annulus}
            \bP\Big(\exists \alpha \in \big\{\rho_{k+1} \le |z| \le \rho_k \big\}  \, : \, f_n(\alpha)=0 \quad  \text{and} \quad |f_n^\prime(\alpha)|\le 1 \Big) \lesssim (1-\rho_k)^{\beta} \, .
        \end{equation}
        Summing~\eqref{eq:bound_on_probability_root_with_small_derivative_per_annulus} over $0\le k\le N$, we get by the union bound that
        \begin{align*}
          \bP\Big(\exists \alpha \in \Big\{1-\delta \le |z| \le 1 - \frac{\log^3(n)}{n} \Big\}  \, : \, f_n(\alpha)=0 \quad  \text{and} \quad &|f_n^\prime(\alpha)|\le 1 \Big) \\ &\lesssim \sum_{k=0}^{N} (1-\rho_k)^\beta \lesssim \Big(\frac{2^N}{n}\Big)^\beta \lesssim \delta^\beta\, ,
        \end{align*}
        as desired.
    \end{proof}
    
    \begin{proof}[Proof of Lemma~\ref{lemma:sum_in_the_annulus_weak}]
        Recall the definition~\eqref{eq:G-def} of the event $\mathcal{G}$, and consider the events
        \begin{align*}
            \cB_1 &= \Big\{ \exists \alpha \in \Big\{1-\delta \le |z| \le 1 - \frac{\log^3(n)}{n} \Big\}  \, : \, f_n(\alpha)=0 \quad  \text{and} \quad |f_n^\prime(\alpha)|\le 1\Big\} \, , \\ \cB_2 &= \Big\{ |\xi_0| \le \frac{1}{n} \Big\} \, .
        \end{align*}
        On the event $\cG\cap \cB_1^c$, for each root $\alpha\in\{1-\delta\le |z|\le 1\}$ we have
        \[
        1\le |f_n^\prime(\alpha)| \le n^{3/2} \log n \, .
        \]
        Furthermore, by Claim~\ref{claim:Jensen_bound_on_number_of_roots}, on the event $\mathcal{G}\cap \cB_2^c$ we have
        \[
        \# \Big\{ |\alpha|\le 1-\frac{\log^3 n}{n} \, : \, f_n(\alpha) = 0 \Big\} \lesssim \frac{n}{\log^3n} \, \log\Big(\frac{n}{|\xi_0|}\Big) \lesssim \frac{n}{\log^2 n} \, .
        \]
        Hence, on the event $\cG\cap \cB_1^c\cap \cB_2^c$ we have
        \begin{equation*}
            \Bigg| \sum_{1-\delta\le |\alpha|\le 1-\log^3 n /n} \log \bigg|\frac{f_n^\prime(\alpha)}{n^{3/2}}\bigg|\, \Bigg| \lesssim \frac{n}{\log^2 n} \log n \le \frac{n}{\sqrt{\log n}}
        \end{equation*}
        for all $n$ large enough. To conclude the proof, it remains to note that as $n\to\infty$ we have $\bP(\cG^c) = o(1)$ by Lemma~\ref{lemma:Salem-Zygmund}, $\bP(\cB_2) = o(1)$ by our assumption that $\xi$ has no atom at $0$, and $\bP(\cB_1) \le C \delta^\beta$ by Claim~\ref{claim:no-small-deriv-weak}. 
    \end{proof}
    
    \subsection{Roots on which the derivative is small II}
        By combining Lemma~\ref{lemma:sum-inside-unit-disk} and Lemma~\ref{lemma:sum_in_the_annulus_weak}, we have essentially proved~\eqref{eq:intro_sum_in_disk_is_negligable}. Therefore, to complete the proof of Proposition~\ref{prop:sum-concentrated}, we want to show that
        \begin{equation}
        \label{eq:sum_concentrated_but_only_in_the_annulus}
        \frac{1}{n} \sum_{1- \log^3 n / n < |\alpha| < 1} \log \left| \frac{f_n'(\alpha)}{n^{3/2}}\right| \xrightarrow[n\to \infty]{\bP} {\sf c}_\ast \, , 
        \end{equation}
        which we do in Section~\ref{sec:lln_for_sum_in_annulus} below. Denote by
        \begin{equation}
            \label{eq:def_of_annulus_A}
            \mathcal{A} = \Big\{z\, : \, 1-\frac{\log^3(n)}{n} \le |z| \le 1 \Big\} \, .
        \end{equation}
        Before we move on, we collect another lemma, which basically shows that we can restrict the sum in~\eqref{eq:sum_concentrated_but_only_in_the_annulus} only to those roots on which the derivative is not too small. 
        \begin{lemma}\label{lemma:annulus-small-deriv-contribution}
        We have
        \begin{equation*}
            \frac{1}{n} \sum_{\alpha\in \mathcal{A}}\log_{-}\left|\frac{f_n'(\alpha)}{n^{3/2}} \right| \one\left\{\frac{|f'(\alpha)|}{n^{3/2}} \leq \frac{1}{\log^{10} n} \right\} \xrightarrow[n \to \infty]{\bP} 0 \, .
        \end{equation*}
    \end{lemma}
    We again separate out a claim which imports the relevant bound from~\cite{Michelen-Yakir-root-separation}. 
    \begin{claim}
        \label{claim:no_small_derivative_in_annulus}
        With $\mathcal{A}$ given by~\eqref{eq:def_of_annulus_A}, we have
        \[
        \lim_{n\to\infty} \bP\Big(\exists \alpha\in \mathcal{A} \, : \, f_n(\alpha) = 0 \quad  \text{and} \quad |f_n^\prime(\alpha)| \le \frac{n^{5/4}}{\log^4 n} \Big) = 0 \, .
        \]
    \end{claim}
    \begin{proof}
        We again yield to~\cite[Lemma~4.3]{Michelen-Yakir-root-separation}, and argue similarly as in the proof of Claim~\ref{claim:no-small-deriv-weak}. Indeed, for $k=0,1,\ldots,\widetilde  N$ with $\widetilde{N} = 1+\lfloor \log \big(\log^3 n\big)\rfloor$, we denote by
        \[
        \widetilde \rho_k = 1-\frac{2^{k}}{n} \, .
        \]
        Then $\widetilde \rho_0 = 1-n^{-1}$ and $\widetilde \rho_{\widetilde{N}} \le 1- \log^3n/n$. 
        By~\cite[Lemma~4.3]{Michelen-Yakir-root-separation}
        we have
               \begin{multline*}
            \bP\Big(\exists \alpha \in \big\{\widetilde \rho_{k+1} \le |z| \le \widetilde\rho_k \big\}  \, : \, f_n(\alpha)=0 \quad  \text{and} \quad |f_n^\prime(\alpha)|\le \frac{n^{5/4}}{\log^4n} \Big) \\ \le \bP\Big(\exists \alpha \in \big\{\widetilde \rho_{k+1} \le |z| \le \widetilde\rho_k \big\}  \, : \, f_n(\alpha)=0 \quad  \text{and} \quad |f_n^\prime(\alpha)|\le \frac{(1-\widetilde{\rho_k})^{-5/4}}{\log n} \Big) \lesssim \frac{1}{\log n} \, .
        \end{multline*}
        Furthermore, another application of~\cite[Lemma~4.3]{Michelen-Yakir-root-separation} (the case $k=-1$ therein) shows that
        \begin{equation*}
            \bP\Big(\exists \alpha\in \Big\{1-\frac{1}{n} \le |z|\le 1 \Big\} \, : \, f_n(\alpha) = 0 \quad  \text{and} \quad |f_n^\prime(\alpha)| \le \frac{n^{5/4}}{\log^4 n} \Big) \lesssim \frac{1}{\log^4 n} \, . 
        \end{equation*}
        By union bounding over the concentric annuli we have
        \begin{align*}
             \bP\Big(\exists &\alpha\in \mathcal{A}  \, : \, f_n(\alpha) = 0 \quad  \text{and} \quad |f_n^\prime(\alpha)| \le \frac{n^{5/4}}{\log^4 n} \Big) \\ & \le \bP\Big(\exists \alpha\in \Big\{1-\frac{1}{n} \le |z|\le 1 \Big\} \, : \, f_n(\alpha) = 0 \quad  \text{and} \quad |f_n^\prime(\alpha)| \le \frac{n^{5/4}}{\log^4 n} \Big) \\ & \qquad \qquad +  \sum_{k=0}^{\widetilde{N}}  \bP\Big(\exists \alpha \in \big\{\widetilde \rho_{k+1} \le |z| \le \widetilde\rho_k \big\}  \, : \, f_n(\alpha)=0 \quad  \text{and} \quad |f_n^\prime(\alpha)|\le \frac{n^{5/4}}{\log^4n} \Big)  \lesssim \frac{\log \log n}{\log n} \, ,
        \end{align*}
        as desired.
    \end{proof}

    Denote by 
    \begin{equation}
    \label{eq:def_of_almost_real_R}
      \mathcal{S} = \Big\{ z = re^{\I\theta} \in \mathcal{A}  \, : \, |\theta| \le \frac{1}{\sqrt{n}} \quad \text{or} \quad |\theta-\pi | \le \frac{1}{\sqrt{n}} \Big\} \, .  
    \end{equation}
    Before we proceed to prove Lemma~\ref{lemma:annulus-small-deriv-contribution}, we state another claim which basically says that roots in the ``almost real" sector $\cR$ do not contribute to the sum~\eqref{eq:sum_concentrated_but_only_in_the_annulus}.
    \begin{claim}
        \label{claim:erdos_turan_bound_on_almost_real_roots}
        Let $\mathcal{S}$ be given by~\eqref{eq:def_of_almost_real_R}, then
        \[
        \lim_{n\to\infty} \bP\Big( \# \{\alpha\in \mathcal{S} \, : \, f_n(\alpha)=0 \} \le \sqrt{n} \, (\log n)  \Big) = 1 \, .
        \]
    \end{claim}
    \begin{proof}
        By the Erd\H{o}s-Tur\'an inequality~\cite[Theorem~1]{Erdos-Turan-AOM} we have 
        \[
        \# \{\alpha\in \mathcal{S}\, : \, f_n(\alpha)=0 \} \le \frac{2}{\pi} \sqrt{n} + 16 \sqrt{n \, \log \bigg(\frac{\max_{z\in \bS^1} |f_n(z)|}{\sqrt{|\xi_0| |\xi_n|}}\bigg) } \, .
        \]
        Therefore, on the event $\mathcal{G} \cap \{|\xi_0| \ge n^{-1}\} \cap \{|\xi_n| \ge n^{-1} \}$ we have 
        \[
        \# \{\alpha\in \mathcal{S} \, : \, f_n(\alpha)=0 \} \lesssim \sqrt{n \log(n^{3/2}\log n)} \lesssim \sqrt{n \log n } \, , 
        \]
        which yields the claim, in view of Lemma~\ref{lemma:Salem-Zygmund} and our assumption of no atom of $\xi$ at the origin. 
    \end{proof}
    \begin{proof}[Proof of Lemma~\ref{lemma:annulus-small-deriv-contribution}]
     Combining Claim~\ref{claim:no_small_derivative_in_annulus} and Claim~\ref{claim:erdos_turan_bound_on_almost_real_roots}, the lemma will follow once we show that
    \begin{equation*}
            \frac{1}{n} \sum_{\alpha\in \mathcal{A}\setminus \mathcal{S}}\log_{-}\left|\frac{f_n'(\alpha)}{n^{3/2}} \right| \one\left\{\frac{1}{n^{1/4}\log^4 n} \le \frac{|f'(\alpha)|}{n^{3/2}} \leq \frac{1}{\log^{10} n} \right\} \xrightarrow[n \to \infty]{\bP} 0 \, .
        \end{equation*}
     which, in turn, would immediately follow from 
    \begin{equation}
    \label{eq:annulus_small_derivative_bound_on_expected_number_of_roots}
        \bE \bigg[ \# \Big\{\alpha\in \mathcal{A} \setminus \mathcal{S} \, : \, f_n(\alpha) = 0 \quad \text{and} \quad \frac{1}{n^{1/4}\log^4 n} \le \frac{|f'(\alpha)|}{n^{3/2}} \leq \frac{1}{\log^{10} n} \Big\} \bigg] \lesssim \frac{n}{\log^2 n} \, . 
    \end{equation}
    This bound is established via a net argument. Let $\kappa = n^{-5/4}\cdot(\log n)^{-6}$ and let ${\sf N}$ be a $\kappa$-net in $\cA\setminus \mathcal{S}$ such that $|{\sf N}|\le 10 \, m(\cA) \, \kappa^{-2} \lesssim n^{3/2} \log^{15} n$. We first note that, on the event $\cG$, there cannot be two distinct roots from the left-hand side of~\eqref{eq:annulus_small_derivative_bound_on_expected_number_of_roots} which are at distance $\le \kappa$. Indeed, suppose that $\alpha\not=\alpha^\prime$ are two such roots, then Taylor's theorem implies that
    \[
    f_n(\alpha^\prime) = f_n(\alpha) + (\alpha^\prime -\alpha) f_n^\prime(\alpha) + O\Big(\frac{|\alpha-\alpha^\prime|^2}{2} n^{5/2} \log n\Big) \, ,
    \]
    and hence
    \[
    |f_n^\prime(\alpha)| \lesssim \kappa n^{5/2} \log n  = \frac{n^{5/4}}{\log^5 n} \, .   
    \]
    That is, such pairs $\{\alpha,\alpha^\prime\}$ are not counted in~\eqref{eq:annulus_small_derivative_bound_on_expected_number_of_roots}, and we get that
    \begin{multline}
    \label{eq:bounding_expected_number_of_bad_roots_by_count_on_net}
    \one_\cG\cdot  \# \Big\{\alpha\in \mathcal{A} \setminus \mathcal{S} \, : \, f_n(\alpha) = 0 \quad \text{and} \quad \frac{1}{n^{1/4}\log^4 n} \le \frac{|f'(\alpha)|}{n^{3/2}} \leq \frac{1}{\log^{10} n} \Big\}  \\  \le \sum_{z\in {\sf N}} \one_\cG\cdot  \one \Big\{ \exists \alpha\in \bD(z,\kappa) \, : \, f_n(\alpha) = 0 \quad \text{and} \quad \frac{|f'(\alpha)|}{n^{3/2}} \leq \frac{1}{\log^{10} n} \Big\} \, .
    \end{multline}
    We now estimate the probabilities appearing on the right-hand side of~\eqref{eq:bounding_expected_number_of_bad_roots_by_count_on_net}. Indeed, on the event $\cG$, if there exists $\alpha\in \bD(z,\kappa)$ with $f_n(\alpha)=0$ and $|f_n^\prime(\alpha)| \le n^{3/2}/\log^{10}n$, then
    \[
    |f_n^\prime(z)| \le |f_n^\prime(\alpha)| + O\big(\kappa n^{5/2} \log n \big) \le 2 \,  \frac{n^{3/2}}{\log^{10}n}
    \]
    and 
    \[
    |f_n(z)| \le \kappa n^{3/2} \log n \le  \frac{n^{1/4}}{\log^2 n} \, . 
    \]
    We see that
    \begin{multline*}
        \cG \cap \Big\{ \exists \alpha\in \bD(z,\kappa) \, : \, f_n(\alpha) = 0 \quad \text{and} \quad \frac{|f'(\alpha)|}{n^{3/2}} \leq \frac{1}{\log^{10} n} \Big\} \\ \subset \Big\{ \frac{|f_n(z)|}{n^{1/2}} \le \frac{1}{n^{1/4} \log^2 n} \, , \ \frac{|f_n^\prime(z)|}{n^{3/2}} \le \frac{1}{\log^{10}n}\Big\} \, ,
    \end{multline*}
    and a standard small-ball bound for $\big(f_n(z),f_n^\prime(z)\big)$ (e.g.~\cite[Claim~3.5]{Michelen-Yakir-root-separation}) implies that 
    \[
    \bP\bigg( \frac{|f_n(z)|}{n^{1/2}} \le \frac{1}{n^{1/4} \log^2 n} \, , \ \frac{|f_n^\prime(z)|}{n^{3/2}} \le \frac{1}{\log^{10}n}  \bigg) \lesssim  \frac{1}{\sqrt{n} \log^{24}n} \, .
    \]
    Plugging the above into~\eqref{eq:bounding_expected_number_of_bad_roots_by_count_on_net} yields that
    \begin{align*}
     \bE \bigg[ \# \Big\{\alpha\in \mathcal{A} \setminus \mathcal{S} \, : \, f_n(\alpha) = 0 \quad \text{and} & \quad \frac{1}{n^{1/4}\log n} \le \frac{|f'(\alpha)|}{n^{3/2}} \leq \frac{1}{\log^{10} n} \Big\} \bigg]  \\ &\lesssim  n \, \bP(\cG^c) + |{\sf N}| \frac{1}{\sqrt{n} \log^{24}n} \stackrel{\text{Lemma}~\ref{lemma:Salem-Zygmund}}{\lesssim} n \, e^{-c\log^2 n} + \frac{n}{\log^{7}n}  \, .  
    \end{align*}
     Hence,~\eqref{eq:annulus_small_derivative_bound_on_expected_number_of_roots} is proved (with a $\polylog(n)$ room to spare) and the proof of the lemma is complete. 
    \end{proof}

    \section{Law of large numbers for sum in the annulus}
    \label{sec:lln_for_sum_in_annulus}
    The goal of this section is to complete the proof of Proposition~\ref{prop:sum-concentrated}, which we will do except for the actual computation of the limit, which we postpone to Section~\ref{sec:computing_the_mean}. Recall that $\mathcal{A}$ the the annulus~\eqref{eq:def_of_annulus_A} and that $\mathcal{S}$ is the ``almost real" sector~\eqref{eq:def_of_almost_real_R}. In view of reductions preformed in Section~\ref{sec:negligible_contributions} above, we restrict our attention to the sum
    \begin{equation}
    \label{eq:main_part_in_sum_in_annulus}
        \frac{1}{n} \sum_{\alpha\in \mathcal{A}\setminus \mathcal{S}} \log\Big|\frac{f_n^\prime(\alpha)}{n^{3/2}}\Big| \,  \one\Big\{\frac{|f_n^\prime(\alpha)|}{n^{3/2} } \ge \frac{1}{\log^{10}n}\Big\} \, .
    \end{equation}
    Let $\mathbb{H} = \big\{ z\in \bC \, : \, \text{Im}(z)>0 \big\} $ be the upper half-plane. Since $f_n$ has real coefficients, the sum~\eqref{eq:main_part_in_sum_in_annulus} consists of two equal contributions, from $\mathbb{H}$ and from $\bC\setminus \mathbb{H}$, and it is equal to
    \begin{equation*}
        \frac{2}{n} \sum_{\alpha\in \mathbb{H}\cap \mathcal{A}\setminus \mathcal{S}} \log\Big|\frac{f_n^\prime(\alpha)}{n^{3/2}}\Big| \,  \one\Big\{\frac{|f_n^\prime(\alpha)|}{n^{3/2} } \ge \frac{1}{\log^{10}n}\Big\} \, .
    \end{equation*}
    We will show that the sum above is in fact close (in $L^1(\bP)$) to a discrete sum over a net in the annulus $\cA$. We will then show, by comparing the original polynomial to a Gaussian polynomial in the annulus, that the corresponding net sum is concentrated as $n\to \infty$.
    \subsection{Setting up the net argument}
    We will take $\beta > 0$ to be a sufficiently small constant ($\beta = 10^{-3}$ will be good enough). For $\eta = n^{-1-\beta}$ and $r_0 = 1 - \log^3 n / n$ we set 
    \[
    M_1 = \lfloor  10 \eta^{-1} \rfloor\, \qquad \text{and} \qquad M_2 = \lfloor 10 \, \eta^{-1} \log (r_0) \rfloor \, .  
    \]
    The net $\mathcal{N}$ is defined via
    \begin{equation*}
            \cN = \bigg\{   \exp\Big( \frac{2b - 1}{2M_2} \cdot \log(r_0) + \I \pi \frac{2a - 1}{2M_1}  \Big) \  : \quad   {\Large \substack{ a \in \{\lfloor n^{-1/2} M_1 \rfloor,\ldots,M_1 - \lfloor n^{-1/2} M_1 \rfloor \} \\[0.2em] b \in \{1,\ldots,M_2 \} } } \, \bigg\}\,.
    \end{equation*}
    In words, $\mathcal{N}$ is an $\eta$-net in $\cA\setminus \mathcal{S}$, and furthermore $|\mathcal{N}| \le \widetilde{O}(n^{1+2\beta})$. For each $z\in \mathcal{N}$ we denote the associated polar rectangle by
    $$\cR_z = \left\{ z \cdot \exp\left( \frac{t}{2M_2} \cdot \log(r_0) + \I \pi  \frac{s}{2M_1}  \right) : t,s \in [-1,1] \right\}\,.  $$
    We also define
    $$\cR_z^+ = \big\{w : \mathrm{dist}(w,\cR_z) \leq n^{-1-3\beta/2} \big\} \quad \text{ and }\quad \cR_z^- = \big\{w \in \cR_z : \mathrm{dist}(w,\cR_z^c) \geq n^{-1-3\beta/2}\big\}\,.$$ 
    For $z\in \mathcal{N}$, we will consider the events 
    \begin{equation}
        \label{eq:def_of_events_C_pm}
        \cC_z^{\pm} := \left\{z - \frac{f_n(z)}{f_n'(z)} \in \cR_z^{\pm}\right\} \cap \left\{\frac{|f_n'(z)|}{n^{3/2}} \geq \frac{1}{\log^{10 \mp 1} n} \right\} \, . 
    \end{equation}
    Roughly speaking, the event~\eqref{eq:def_of_events_C_pm} occurs if the linear approximation of $f_n$ at the point $z$ predicts a root in the rectangle $\cR_z$, while the derivative is just of typical size. The next claim is a simple consequence of this observation.
    
    \begin{claim}
        \label{claim:sum_of_probabilities_over_net_points_is_bounded}
        We have
        $
        \displaystyle \frac{1}{n} \sum_{z\in \cN} \bP\big(\cC_z^+ \big) = \widetilde{O }(1) \, . 
        $
    \end{claim}
    \noindent

    The proof of Claim~\ref{claim:sum_of_probabilities_over_net_points_is_bounded} is postponed to Section~\ref{sec:computing_the_mean} below, where it is derived as a consequence of a more general statement (see Lemma~\ref{lemma:comparison_with_a_gaussian_on_the_event_C_pm} below). We note that in fact this sum is $O(1)$ (the proof of Lemma~\ref{lemma:computation_of_the_mean_for_sum_over_net} implicitly shows this), but for its application in the next Lemma, a bound of $\widetilde{O}(1)$ is good enough. The next lemma shows how the sum over the net points captures the asymptotic size of~\eqref{eq:main_part_in_sum_in_annulus}.  We note that the factor of two appearing in the net sum is due to the fact that $\cN$ is a net in the upper half plane and the contribution from the lower half plane is identical.

    \begin{lemma}\label{lemma:two_sided_bound_roots_vs_net}
    On the event $\cG$ given by~\eqref{eq:G-def}, for $\bullet \in \{-,+\}$ we have 
      \begin{equation*}
          \frac{1}{n} \sum_{\alpha \in \cA\setminus \mathcal{S}} \log_\bullet \Big|\frac{f_n'(\alpha)}{n^{3/2}} \Big| \, \one\Big\{\frac{|f_n'(\alpha)|}{n^{3/2}} \geq \frac{1}{\log^{10}n}  \Big\} \leq \frac{2}{n} \sum_{z \in \cN} \log_\bullet \Big|\frac{f_n'(z)}{n^{3/2}} \Big| \cdot \one_{\cC_z^+} + O(n^{-\beta/2}) \, ,
      \end{equation*}  
      and \begin{equation*}
\frac{1}{n} \sum_{\alpha \in \cA\setminus \mathcal{S}} \log_\bullet\Big|\frac{f_n'(\alpha)}{n^{3/2}} \Big| \,  \one\Big\{\frac{|f_n'(\alpha)|}{n^{3/2}} \geq \frac{1}{\log^{10}n}  \Big\} \geq \frac{2}{n} \sum_{z \in \cN}  \log_\bullet\Big|\frac{f_n'(z)}{n^{3/2}} \Big| \cdot  \one_{\cC_z^-} + O(n^{-\beta/2})\,.
      \end{equation*}
    \end{lemma}    
    \begin{proof}
        We start with the upper bound.  Let $\mathbb{H} = \big\{z \, :  \text{Im}(z)>0\big\}$ denote the upper half-plane. Clearly, each root $\alpha\in (\cA\setminus \mathcal{S})  \cap \mathbb{H} $ is contained in $\cR_z$ for some $z\in \cN$.  We first show that, on the event $\cG$, the existence of a root $\alpha\in \cR_z$ with $|f_n^\prime(\alpha)|\ge n^{3/2}/\log^{10} n$ implies that $\cC_z^+$ holds. Obviously
        \begin{equation}
            \label{eq:net_point_and_roots_have_comparable_derivative_under_G}
            |f_n^\prime(z)| \ge |f_n^\prime(\alpha)| - \widetilde{O}(\eta n^{5/2}) \ge |f_n^\prime(\alpha)|\, \big(1 + \widetilde{O}(n^{-\beta})\big) \, .
        \end{equation} 
        Furthermore, Taylor's theorem implies that 
        \[
        0 = f_n(\alpha) = f_n(z) + (\alpha-z)f_n^\prime(z) + \widetilde{O}(\eta^2 n^{5/2}) \, , 
        \]
        and hence~\eqref{eq:net_point_and_roots_have_comparable_derivative_under_G} shows that
        \[
        z-\frac{f_n(z)}{f_n^\prime(z)} = \alpha + \widetilde{O}(n\eta^2) = \alpha + \widetilde{O}(n^{-1-2\beta}) \, .
        \]
        That is, the event $\cC_z^+$ holds. It remains to note that, on the event $\cG$, there are no two distinct roots $\{\alpha,\alpha^\prime\}$ from the sum~\eqref{eq:main_part_in_sum_in_annulus} in the same $\cR_z$. Otherwise, we would have
        \[
        f_n(\alpha^\prime) =f_n(\alpha) + (\alpha^\prime-\alpha) f_n^\prime(\alpha) + \widetilde{O}\big(n^{5/2}|\alpha^\prime-\alpha|^2\big) \, ,
        \]
        which gives $|f_n^\prime(\alpha)| = \widetilde{O}(n^{3/2-\beta})$, which is a contradiction. Altogether, we get that
        \begin{align*}
            \frac{1}{n} \sum_{\alpha \in \cA\setminus \mathcal{S}} & \log_\bullet \Big|\frac{f_n'(\alpha)}{n^{3/2}} \Big| \, \one\Big\{\frac{|f_n'(\alpha)|}{n^{3/2}} \geq \frac{1}{\log^{10}n}  \Big\} \\ &\leq \frac{2}{n} \sum_{z\in \cN} \log_\bullet \Big|\frac{f^\prime(z)}{n^{3/2}}\big(1 + \widetilde{O}(n^{-\beta})\big)\Big| \cdot \one_{\cC_z^+}  \\ &\leq \frac{2}{n} \sum_{z\in \cN} \log_\bullet \Big|\frac{f^\prime(z)}{n^{3/2}}\Big| \cdot \one_{\cC_z^+} + \widetilde{O}\Big( \frac{n^{-\beta}}{n}\sum_{z\in \cN} \bP\big(\cC_z^+\big)\Big) \stackrel{\text{Claim}~\ref{claim:sum_of_probabilities_over_net_points_is_bounded}}{\le} \frac{2}{n} \sum_{z\in \cN} \log_\bullet \Big|\frac{f^\prime(z)}{n^{3/2}}\Big| \cdot \one_{\cC_z^+} + \widetilde{O}(n^{-\beta}) \, .       
            \end{align*}
            To prove the corresponding lower bound, we will first show that if $\cG\cap \cC_z^-$ occurs, then necessarily there exists $\alpha\in \cR_z$ with $f_n(\alpha) = 0$. Indeed, define the linear approximation $L_z(\zeta) = f_n(z) + (\zeta-z)f_n^\prime(z)$ around $z$ and note that $L_z(\zeta_0) = 0$, where
            \[
            \zeta_0 = z-\frac{f_n(z)}{f_n^\prime(z)}. 
            \]
            On the event $\cC_z^-$, we have $\zeta_0\in \cR_z^-$. To show there is actually a root of $f_n$ nearby, we will turn to Rouch\'e's theorem. Indeed, let $\gamma = n^{-1-7\beta/4} = \eta n^{3\beta/4} $ and note that 
            \[
            |L_z(\zeta_0 + \gamma e^{\I\theta} )| = \gamma |f_n^\prime(z)| \ge \widetilde{\Omega}(\gamma n^{3/2}) = \widetilde{\Omega}(n^{1/2 - 7\beta/4}) \, 
            \]
            for all $\theta\in[0,2\pi]$. On the event $\cG$, Taylor's theorem allows us to bound the difference
            \[
            \big| L_z(\zeta_0 + \gamma e^{\I\theta} ) -f_n(\zeta_0 +\gamma e^{\I\theta}) \big| \le \widetilde{O}(\eta^2 n^{5/2}) = \widetilde{O}(n^{1/2-2\beta}) \, .
            \]
            Hence, Rouch\'e's theorem implies that $f_n$ and $L_z$ both have a single root in $\bD(\zeta_0,\gamma)\subset \cR_z$. It remains to perform a simple Taylor expansion (the same one as in~\eqref{eq:net_point_and_roots_have_comparable_derivative_under_G}) to show that at this root, we have $|f_n^\prime(\alpha)|\ge n^{3/2}/\log^{10}n$. Since the interiors of $\{\cR_z\}_{z\in \cN}$ are disjoint, we arrive at
            \begin{multline*}
                \frac{2}{n} \sum_{z \in \cN}  \log_\bullet\Big|\frac{f_n'(z)}{n^{3/2}} \Big| \cdot  \one_{\cC_z^-} \le \frac{1}{n} \sum_{\alpha \in \cA\setminus \mathcal{S}} \log_\bullet\Big|\frac{f_n'(\alpha)}{n^{3/2}}\big(1+\widetilde{O}(n^{-\beta})\big) \Big| \,  \one\Big\{\frac{|f_n'(\alpha)|}{n^{3/2}} \geq \frac{1}{\log^{10}n}  \Big\}  \\ \le \frac{1}{n} \sum_{\alpha \in \cA\setminus \mathcal{S}} \log_\bullet\Big|\frac{f_n'(\alpha)}{n^{3/2}}\Big| \,  \one\Big\{\frac{|f_n'(\alpha)|}{n^{3/2}} \geq \frac{1}{\log^{10}n}  \Big\} + \widetilde{O}(n^{-\beta})
            \end{multline*}
            which is the desired lower bound, and we are done. 
    \end{proof}
    
    \subsection{Gaussian comparison}
    \label{subsec:gaussian_comparison}
    With Lemma~\ref{lemma:two_sided_bound_roots_vs_net}, we see that to prove concentration of~\eqref{eq:main_part_in_sum_in_annulus} we need to understand the corresponding net sums over $\mathcal{N}$. This is achievable as we know that the random polynomial $f_n$ (and its derivative) evaluated at finitely many separated points is roughly Gaussian. To better explain this point, let 
    \begin{equation}
    \label{eq:def_of_random_polynomial_with_gaussian_coefficients}
        g_n(z) = \sum_{k=0}^{n} \gamma_k \, z^k \, , 
    \end{equation}
    be a random Kac polynomial with $\gamma_0,\ldots,\gamma_n$ being i.i.d. standard complex Gaussian random variables\footnote{We say that $Z= X+ \I Y$ is a standard complex Gaussian if $X,Y$ are independent Gaussian random variables with mean zero and variance $1/2$. }. We denote by $\bP_g, \, \bE_g$, the corresponding probability measure and expectation operator. The next lemma is a consequence of the Berry-Esseen theorem:
    \begin{lemma}
        \label{lemma:berry_esseen_in_annulus}
        Let $z\in \cA\setminus \mathcal{S}$, then
        \begin{equation*}
        \sup_{K \subset \bC^2}\Big| \P\Big( \big(f_n(z),f_n'(z)\big) \in K \Big) - \P_g\Big( \big(g_n(z),g_n'(z)\big) \in K \Big) \Big| \leq 
        \Otilde(n^{-1/2})
        \end{equation*}
        where the supremum is over all convex sets $K\subset \bC^2$. Furthermore, for $z,w\in \cA\setminus \mathcal{S}$ such that $|z-w|\ge n^{-1/2}$, we have
        \begin{equation*}
        \sup_{K \subset \bC^4}\Big| \P\Big( \big(f_n(z),f_n'(z),f_n(w),f_n'(w)\big) \in K \Big) - \P_g\Big( \big(g_n(z),g_n'(z),g_n(w),g_n'(w)\big) \in K \Big) \Big| \leq 
        \Otilde(n^{-1/2})
        \end{equation*}
        where here the supremum is over all convex sets $K\subset \bC^4$. 
        \end{lemma}
    Lemma~\ref{lemma:berry_esseen_in_annulus} is deduced from a standard, sufficiently general Berry-Esseen theorem in Section~\ref{sec:gaussian_comparison}. As a first application, we prove the following simple claim, showing that the events $\cC_z^+$ and $\cC_z^-$ are close.
    \begin{claim}
        \label{claim:difference_between_C_+and_C_-_small}
        For all $z\in \cN$ we have 
        \[
        \bP\big( \cC_z^+ \setminus \cC_z^- \big) \le n^{-9\beta/4} \, . 
        \]
    \end{claim}
    \begin{proof}
        We first note that 
        \begin{align*}
            \bP\big( \cC_z^+ \setminus \cC_z^- \big) &= \bP\Big(z-\frac{f_n(z)}{f_n^\prime(z)} \in \cR_z^+\setminus \cR_z^-\Big) \\ & \le \bP\Big(-\frac{f_n(z)}{\sqrt{n}} \in n\log^2 n\cdot\big(\cR_z^+\setminus \cR_z^- - z\big)\Big) + e^{-c\log^2 n} \, ,
        \end{align*}
        where the inequality is just Lemma~\ref{lemma:Salem-Zygmund}. We bound the Lebesgue measure of the set in this event by
        \[
        m\Big(n\log^2 n\cdot\big(\cR_z^+\setminus \cR_z^- - z\big)\Big) \lesssim n^2 \log^4 n \, \big(n^{-1-\beta}\big)^2 n^{-\beta/2} \le \widetilde{O}(n^{-5\beta/2}) \, .
        \]
        Covering this set with a constant number of disks with radii $r = \widetilde{O}(n^{-5\beta/4})$, we can apply Lemma~\ref{lemma:berry_esseen_in_annulus} and the union bound to get
        \[
        \bP\Big(-\frac{f_n(z)}{\sqrt{n}} \in n\log^2 n\cdot\big(\cR_z^+\setminus \cR_z^- - z\big)\Big) \lesssim \sup_{x\in \bC} \, \bP_g\Big(\frac{g
        _n(z)}{\sqrt{n}} \in \bD(x, r) \Big)  + \widetilde{O}(n^{-1/2}) = \widetilde{O}(n^{-5\beta/2})
        \]
        where the last inequality is due to the fact that the density of the Gaussian $g_n(z)/\sqrt{n}$ is $\widetilde{O}(1)$. 
    \end{proof}
    A simple corollary of the above is a bound on the $L^1(\bP)$ distance between the sum~\eqref{eq:main_part_in_sum_in_annulus} and the sum over net points. 
    \begin{lemma}
        \label{lemma:sum_over_roots_and_sum_over_net_are_close_in_L_1}
        For $\bullet \in \{-,+\}$ we have
        \begin{equation*}
            \bE \, \bigg| \frac{1}{n} \sum_{\alpha \in \cA\setminus \mathcal{S}} \log_\bullet \Big|\frac{f_n'(\alpha)}{n^{3/2}} \Big| \, \one\Big\{\frac{|f_n'(\alpha)|}{n^{3/2}} \geq \frac{1}{\log^{10}n}  \Big\} - \frac{2}{n} \sum_{z \in \cN} \log_\bullet \Big|\frac{f_n'(z)}{n^{3/2}} \Big| \cdot \one_{\cC_z^+} \bigg| \lesssim n^{-\beta/5} \, .
        \end{equation*}
    \end{lemma}
    \begin{proof}
        By Lemma~\ref{lemma:Salem-Zygmund}, we can bound the expectation on the event $\cG$, as both random variables have $L^1(\bP)$ growing at most polynomially in $n$, while $\bP(\cG^c) \le \exp(-c\log^2 n)$. On the event $\cG$, Lemma~\ref{lemma:two_sided_bound_roots_vs_net} shows that the $L^1(\bP)$ distance is bounded by 
        \[
        n^{-\beta/2} + \frac{\log n}{n} \sum_{z\in \cN} \bP\big( \cC_z^+ \setminus \cC_z^- \big) \stackrel{\text{Claim}~\ref{claim:difference_between_C_+and_C_-_small}}{\lesssim} n^{-\beta} + \frac{\log n}{n} \, |\cN| \, n^{-9\beta/4} \lesssim n^{-\beta/5} \, ,
        \]
        as desired. 
    \end{proof}
    \subsection{Mean and variance of the net sum}
    As another application of the Gaussian comparison (Lemma~\ref{lemma:berry_esseen_in_annulus}), we can show that the sum over net points is is concentrated. This is a consequence of the next claim, which states that the net sum decorrelates for separated points.   
    \begin{claim}
        \label{claim:decorrelation_in_annulus}
        For $z,w\in \cN$ such that $|z-w|\ge n^{-1/2}$ we have
        \begin{equation*}
            \bigg| \bE\Big[ \log_\bullet\Big|\frac{f_n^\prime(z)}{n^{3/2}} \cdot \log_\bullet\Big|\frac{f_n^\prime(z)}{n^{3/2}} \Big| \, \one_{\cC_z^+\cap\cC_w^+} \Big] - \bE\Big[ \log_\bullet\Big|\frac{f_n^\prime(z)}{n^{3/2}} \, \one_{\cC_z^+} \Big] \cdot \bE\big[ \log_\bullet\Big|\frac{f_n^\prime(w)}{n^{3/2}} \, \one_{\cC_w^+} \Big]  \bigg| \le n^{-1/20} \, ,
        \end{equation*}
        for both $\bullet\in \{-,+\}$. 
    \end{claim}
    As we already mentioned, Claim~\ref{claim:decorrelation_in_annulus} is a simple consequence of Lemma~\ref{lemma:berry_esseen_in_annulus} and is proved in Section~\ref{sec:gaussian_comparison} below. Yet another consequence of the Gaussian comparison is that we can compute the limiting constant for the sum over the net via the Kac-Rice formula. This will be performed in Section~\ref{sec:computing_the_mean}, but for now we can state the following lemma:
    \begin{lemma}
        \label{lemma:computation_of_the_mean_for_sum_over_net}
        We have
        \[
        \lim_{n\to\infty} \bE\bigg[ \,  \frac{2}{n} \sum_{z \in \cN} \log \Big|\frac{f_n'(z)}{n^{3/2}} \Big| \cdot \one_{\cC_z^+} \bigg] = {\sf c}_\ast \, ,
        \]
        where ${\sf c}_\ast$ is given by~\eqref{eq:def_of_c_ast}.
    \end{lemma}
    
    \subsection{Combining the ingredients: Proof of Proposition~\ref{prop:sum-concentrated}}
    We conclude this section by proving
    \begin{equation}
    \label{eq:sum_concentrated_in_proof}
           \frac{1}{n} \sum_{\alpha\in \bD} \log \left|\frac{f_n'(\alpha)}{n^{3/2}} \right| \xrightarrow[n\to \infty]{\bP} {\sf c}_\ast \, ,
    \end{equation}
    which is exactly the statement of Proposition~\ref{prop:sum-concentrated}. As we are proving convergence in probability, Lemma~\ref{lemma:Salem-Zygmund} asserts that we may assume throughout the event $\cG$ holds. 
    Indeed, towards~\eqref{eq:sum_concentrated_in_proof} we first show that 
    \begin{equation}
        \label{eq:inner_sum_small_in_proof_sum_concentrated}
        \frac{1}{n} \sum_{\alpha\in \bD\setminus \cA} \log\Big|\frac{f_n^\prime(\alpha)}{n^{3/2}}\Big| \xrightarrow[n\to \infty]{\bP} 0 \, .
    \end{equation}
    Fix some $\delta>0$ and observe that for all $\eps>0$ we have
    \begin{multline*}
        \bP\bigg( \, \bigg| \frac{1}{n} \sum_{\alpha\in \bD\setminus \cA} \log\Big|\frac{f_n^\prime(\alpha)}{n^{3/2}}\Big|\bigg| \ge \eps \bigg) \\ \le \bP\bigg( \, \bigg| \frac{1}{n} \sum_{|\alpha| \le 1-\delta} \log\Big|\frac{f_n^\prime(\alpha)}{n^{3/2}}\Big|\bigg| \ge \frac{\eps}{2} \bigg) + \bP\bigg( \, \bigg| \frac{1}{n} \sum_{1-\delta\le |\alpha| \le 1-\log^3 n/n} \log\Big|\frac{f_n^\prime(\alpha)}{n^{3/2}}\Big|\bigg| \ge \frac{\eps}{2} \bigg) \, . 
    \end{multline*}
    By Lemma~\ref{lemma:sum-inside-unit-disk}, the first term in the sum tends to $0$ as $n\to\infty$, and Lemma~\ref{lemma:sum_in_the_annulus_weak} shows that
    \[
    \limsup_{n\to\infty} \bP\bigg( \, \bigg| \frac{1}{n} \sum_{\alpha\in \bD\setminus \cA} \log\Big|\frac{f_n^\prime(\alpha)}{n^{3/2}}\Big|\bigg| \ge \eps \bigg) \lesssim \delta^\beta \, .
    \]
    As this is true for all $\delta>0$, we conclude that~\eqref{eq:inner_sum_small_in_proof_sum_concentrated} holds. By combining this observation with Lemma~\ref{lemma:annulus-small-deriv-contribution}, it will be sufficient to show that
    \begin{equation*}
           \frac{1}{n} \sum_{\alpha\in \cA} \log \left|\frac{f_n'(\alpha)}{n^{3/2}} \right| \cdot \one\Big\{\frac{|f'(\alpha)|}{n^{3/2}} \ge \frac{1}{\log^{10} n} \Big\}   \xrightarrow[n\to \infty]{\bP} {\sf c}_\ast \, ,
    \end{equation*}
    Moreover, recalling that the sector $\mathcal{S}$ is given by~\eqref{eq:def_of_almost_real_R}, Claim~\ref{claim:erdos_turan_bound_on_almost_real_roots} asserts that with high probability the number of roots in $\mathcal{S}$ is $\widetilde{O}(\sqrt{n})$. Hence, with probability tending to one as $n\to\infty$ we have
    \[
    \frac{1}{n} \sum_{\alpha\in \mathcal{S}} \bigg| \log \Big|\frac{f_n'(\alpha)}{n^{3/2}} \Big| \bigg| \cdot \one\Big\{\frac{|f'(\alpha)|}{n^{3/2}} \ge \frac{1}{\log^{10} n} \Big\} =  \widetilde{O}(n^{-1/2})
    \]
    and we have reduced~\eqref{eq:sum_concentrated_in_proof} to showing that
    \begin{equation}
    \label{eq:sum_concentrated_in_proof_after_reductions}
           \frac{1}{n} \sum_{\alpha\in \cA\setminus \mathcal{S}} \log \left|\frac{f_n'(\alpha)}{n^{3/2}} \right| \cdot \one\Big\{\frac{|f'(\alpha)|}{n^{3/2}} \ge \frac{1}{\log^{10} n} \Big\}   \xrightarrow[n\to \infty]{\bP} {\sf c}_\ast \, .
    \end{equation}
    In view of Lemma~\ref{lemma:sum_over_roots_and_sum_over_net_are_close_in_L_1}, concentration of the sum in~\eqref{eq:sum_concentrated_in_proof_after_reductions} is an immediate consequence of
    \begin{equation}
        \label{eq:bound_on_variance_sum_over_net}
        \text{Var}\bigg( \frac{1}{n} \sum_{z \in \cN} \log \Big|\frac{f_n'(z)}{n^{3/2}} \Big| \cdot \one_{\cC_z^+} \bigg) \lesssim n^{-5\beta} \, .
    \end{equation}
    After~\eqref{eq:bound_on_variance_sum_over_net} is established, the limiting constant is identified via Lemma~\ref{lemma:computation_of_the_mean_for_sum_over_net}. It remains to prove~\eqref{eq:bound_on_variance_sum_over_net} holds. Indeed, we denote by $\cN^2 = \cN\times \cN$ the Cartesian product of the net points, and let 
    \[
    \cN_{g,2} = \Big\{(z,w) \in \cN^2 \, : \,  |z-w| \ge n^{-1/2} \Big\} \, .
    \]
    We further set $\cN_{b,2} = \cN^2 \setminus \cN_{g,2}$, and note that $$|\cN_{b,2}| \le  \widetilde{O}(n^{1+2\beta} \cdot n^{1/2 + 2\beta}) = \widetilde{O}(n^{3/2 + 4\beta}) \, ,$$
    while $|\cN^2| = \widetilde{O}(n^{2+4\beta})$. We have
    \begin{align*}
        \text{Var}\bigg( &\frac{1}{n} \sum_{z \in \cN} \log \Big|\frac{f_n'(z)}{n^{3/2}} \Big| \cdot \one_{\cC_z^+} \bigg) \\ &\le \frac{1}{n^2} \sum_{(z,w)\in \cN^2} \bigg|\bE\Big[ \log_\bullet\Big|\frac{f_n^\prime(z)}{n^{3/2}} \cdot \log_\bullet\Big|\frac{f_n^\prime(z)}{n^{3/2}} \Big| \, \one_{\cC_z^+\cap\cC_w^+} \Big] - \bE\Big[ \log_\bullet\Big|\frac{f_n^\prime(z)}{n^{3/2}} \, \one_{\cC_z^+} \Big] \cdot \bE\big[ \log_\bullet\Big|\frac{f_n^\prime(w)}{n^{3/2}} \, \one_{\cC_w^+} \Big]\bigg| \\ &\le \frac{\log^2 n}{n^2} \, |\cN_{b,2}| + n^{-1/20} \, \frac{|\cN^2|}{n^2} \lesssim n^{-5\beta} \, ,
    \end{align*}
    where the second inequality follows from Claim~\ref{claim:decorrelation_in_annulus} (and a trivial bound for $(z,w)\in \cN_{b,2}$). We have thus proved the bound~\eqref{eq:bound_on_variance_sum_over_net}, and with that we get~\eqref{eq:sum_concentrated_in_proof}.  
    \qed
    
    \section{Computing the mean}
    \label{sec:computing_the_mean}
    \noindent
    Recall that $g_n$ is the Gaussian random polynomial~\eqref{eq:def_of_random_polynomial_with_gaussian_coefficients}, which is the (complex) Gaussian counterpart to the Kac random polynomial $f_n$ we consider throughout the paper. The goal of this section is to prove Lemma~\ref{lemma:computation_of_the_mean_for_sum_over_net}, which we will do by first showing that, as $n\to \infty$
    \begin{equation}
    \label{eq:expectation_for_polynomial_and_gaussian_over_net_is_the_same}
        \bE\bigg[ \,  \frac{1}{n} \sum_{z \in \cN} \log \Big|\frac{f_n'(z)}{n^{3/2}} \Big| \cdot \one_{\cC_z^+} \bigg] = \bE_g\bigg[ \,  \frac{1}{n} \sum_{z \in \cN} \log \Big|\frac{g_n'(z)}{n^{3/2}} \Big| \cdot \one_{\cC_z^+} \bigg] + o(1) \, .
    \end{equation}
    For Gaussian polynomials, we have additional tools to compute the limiting expectation in~\eqref{eq:expectation_for_polynomial_and_gaussian_over_net_is_the_same}. Indeed, we shall apply below the Kac-Rice formula (see Section~\ref{subsection:application_of_kac_rice} below) to compute the Gaussian expectation, and with that complete the proof of Lemma~\ref{lemma:computation_of_the_mean_for_sum_over_net}. 
   
    \subsection{Approximating the event $\cC_z^+$}
    Recall that for $z \in \cN$ the event $\cC_z^+$ is given by~\eqref{eq:def_of_events_C_pm}. The key step towards the proof of~\eqref{eq:expectation_for_polynomial_and_gaussian_over_net_is_the_same} is to show that $\cC_z^+$ has the same first-order asymptotic probability under $\bP$ and $\bP_g$, as $n\to\infty$. A slightly more general formulation of this fact is given by the next lemma. 
    \begin{lemma}
        \label{lemma:comparison_with_a_gaussian_on_the_event_C_pm}
        For $z\in \cN$ we have
        \begin{equation}
        \label{eq:lemma_comparison_with_a_gaussian_on_the_event_C_pm_first}
            \Big| \bP(\cC_z^+) - \bP_g(\cC_z^+) \Big| \le n^{-1/20} \, .
        \end{equation}
        Furthermore, for both $\bullet\in\{-,+\}$ we have
        \begin{equation}
        \label{eq:lemma_comparison_with_a_gaussian_on_the_event_C_pm_second}
            \bigg| \bE\Big[ \log_\bullet \Big|\frac{f_n^\prime(z)}{n^{3/2}}\Big| \cdot \one_{\cC_z^+} \Big]  - \bE_g\Big[ \log_\bullet \Big|\frac{g_n^\prime(z)}{n^{3/2}}\Big| \cdot \one_{\cC_z^+} \Big] \, \bigg|\le n^{-1/20} \, .
        \end{equation}
        Finally, for $z,w\in \cN$ such that $|z-w|\ge n^{-1/2}$ and for $\bullet\in\{-,+\}$ we have
        \begin{equation}
        \label{eq:lemma_comparison_with_a_gaussian_on_the_event_C_pm_third}
            \bigg| \bE\Big[ \log_\bullet \Big|\frac{f_n^\prime(z)}{n^{3/2}}\Big| \cdot \log_\bullet \Big|\frac{f_n^\prime(w)}{n^{3/2}}\Big| \cdot \one_{\cC_z^+\cap \cC_w^{+}} \Big]  - \bE_g\Big[ \log_\bullet \Big|\frac{g_n^\prime(z)}{n^{3/2}}\Big| \cdot \log_\bullet \Big|\frac{g_n^\prime(w)}{n^{3/2}}\Big| \cdot \one_{\cC_z^+\cap \cC_w^{+}} \Big] \, \bigg|\le n^{-1/20} \, .
        \end{equation}
    \end{lemma}
    \begin{proof}
        We will only prove~\eqref{eq:lemma_comparison_with_a_gaussian_on_the_event_C_pm_third} as both~\eqref{eq:lemma_comparison_with_a_gaussian_on_the_event_C_pm_first} and~\eqref{eq:lemma_comparison_with_a_gaussian_on_the_event_C_pm_second} have very similar (and simpler) proofs. We will also just consider the case $\bullet = +$, as the complementary case $\bullet = -$ is identical. We can intersect both expectations with the event $\cG$ given by~\eqref{eq:G-def}, since Lemma~\ref{lemma:Salem-Zygmund} guarantees that
        \[
        \bE\Big[ \log_\bullet \Big|\frac{f_n^\prime(z)}{n^{3/2}}\Big| \cdot \log_\bullet \Big|\frac{f_n^\prime(w)}{n^{3/2}}\Big| \cdot \one_{\cC_z^+\cap \cC_w^{+}} \cdot \one_{\cG^c} \Big] \le \widetilde{O}\big(\bP(\cG^c)^{1/2}\big) \lesssim e^{-c\log^2 n} \, ,
        \]
        which we can safely absorb in the error term. For $z\in \cN$, we set 
        \begin{equation}
        \label{eq:def_of_E_z}
            \mathcal{E}_z = \Big\{ (X,Y) \in \bC^2 \, : \,  \frac{X}{nY} \in \cR_z^+ -z \, , \, \frac1{\log^9 n} \le |Y| \le \log^2 n \, , \, |X|\le \log^2 n \, \Big\} \, .
        \end{equation}
        That is, $\cE_z$ is the subset of $\bC^2$ which corresponds to the event $\cC_z^+\cap\cG$, with 
        \[
        X = \frac{f_n(z)}{\sqrt{n}} \qquad \text{and} \qquad Y = \frac{f_n^\prime(z)}{n^{3/2}} \, .
        \]
        Recalling that $\cR_z^+$ is a polar rectangle around $z$ with side lengths $\widetilde{O}(n^{-1-\beta})$, we get that $m(\cR_z^+) = \widetilde{O}(n^{-2-2\beta})$ and hence
        \begin{equation}
            \label{eq:bound_on_measure_of_E_z}
            m(\cE_z) \le \widetilde{O}\big(n^2 m(\cR_z^+) \big) = \widetilde{O}(n^{-2\beta}) \, .
        \end{equation}We further denote by $\mu$ the law of
        \[
        \Big(\frac{f_n(z)}{\sqrt{n}},\frac{f_n^\prime(z)}{n^{3/2}},\frac{f_n(w)}{\sqrt{n}},\frac{f_n^\prime(w)}{n^{3/2}}\Big) \qquad \text{on } \ \bC^4\, ,
        \]
        and we get that 
        \begin{multline}
        \label{eq:expectaion_of_logs_in_complex_coordinates}
            \bE\Big[ \log_+ \Big|\frac{f_n^\prime(z)}{n^{3/2}}\Big| \cdot \log_+ \Big|\frac{f_n^\prime(w)}{n^{3/2}}\Big| \cdot \one_{\cC_z^+\cap \cC_w^{+}} \cdot \one_{\cG} \Big] \\  = \iint_{\substack{  (X_1,Y_1)\in \cE_z \\  (X_2,Y_2)\in \cE_w }} \log_+|Y_1| \cdot \log_+|Y_2| \, {\rm d}\mu(X_1,Y_1,X_2,Y_2) \, .
        \end{multline}
        We will also denote by $\mu_g$ the corresponding (Gaussian) law of 
        \[
        \Big(\frac{g_n(z)}{\sqrt{n}},\frac{g_n^\prime(z)}{n^{3/2}},\frac{g_n(w)}{\sqrt{n}},\frac{g_n^\prime(w)}{n^{3/2}}\Big) \qquad \text{on }  \ \bC^4 \, . 
        \]
        In view of~\eqref{eq:expectaion_of_logs_in_complex_coordinates}, we want to show that 
        \begin{multline}
        \label{eq:proof_of_lemma_comparison_for_C_what_we_want_in_complex_coordinates}
            \iint_{\substack{  (X_1,Y_1)\in \cE_z \\  (X_2,Y_2)\in \cE_w }} \log_+|Y_1| \cdot \log_+|Y_2| \, {\rm d}\mu(X_1,Y_1,X_2,Y_2) \\ = \iint_{\substack{  (X_1,Y_1)\in \cE_z \\  (X_2,Y_2)\in \cE_w }} \log_+|Y_1| \cdot \log_+|Y_2| \, {\rm d}\mu_g(X_1,Y_1,X_2,Y_2) + O(n^{-1/20})
        \end{multline}
        which we do in what follows. Let $\kappa=\beta>0$ be a small parameter, and consider the tiling of $\bC^2$ parametrized by the lattice $(n^{-\kappa} 
 \, \bZ)^4$, namely
        \begin{equation*}
            \mathcal{Q} = \Big\{ p + \big[0,n^{-\kappa}]^{4} \, : \, p\in (n^{-\kappa} \, \bZ)^4 \Big\} \, .
        \end{equation*}
        Then all cubes in $\mathcal{Q}$ have mutually disjoint interiors and their union is $\bR^4 \simeq \bC^2$.  For $z\in \cN$ and $\cE_z$ given by~\eqref{eq:def_of_E_z} we set 
        \begin{equation*}
            \Gamma_z = \big\{ Q\in \mathcal{Q} \, : \, Q\subset \cE_z\big\} \, , \quad 
            \Gamma_z^{b} = \big\{ Q\in \mathcal{Q} \, : \, Q\cap \cE_z \not = \emptyset \, , \ Q^c\cap \cE_z \not=\emptyset \big\}  \, , \quad   \overline \Ga_z = \Ga_z \cup \Ga_z^b \, , 
        \end{equation*}
        and note that since $\cE_z$ is a Lipschitz domain,~\eqref{eq:bound_on_measure_of_E_z} implies that 
        \begin{equation}
        \label{eq:bound_on_caridinality_of_Gamma}
            |\Ga_z| \le \widetilde{O}(n^{4\kappa - 2\beta}) \quad \text{and} \quad |\Ga_z^b| \le  \widetilde{O}(n^{3\kappa-2\beta}) \, .
        \end{equation}                
        To prove~\eqref{eq:proof_of_lemma_comparison_for_C_what_we_want_in_complex_coordinates}, we need to compare between the probability measures $\mu$ and $\mu_g$. Indeed, for any pair of cubes $Q_1,Q_2\in \mathcal{Q}$, Lemma~\ref{lemma:berry_esseen_in_annulus} implies that
        \begin{equation}
        \label{eq:comparison_of_measures_on_lattice_cubes_complex_coordinates}
            \mu\big(Q_1\times Q_2\big) = \mu_g\big(Q_1\times Q_2\big) + \widetilde{O}(n^{-1/2}) \, .
        \end{equation}
        Furthermore, since $\mu_g$ is absolutely continuous with respect to Lebesgue measure in $\bC^4$ with a bounded (Gaussian) density, we also have
        \begin{equation}
        \label{eq:bound_on_gaussian_meausure_on_lattice_cubes}
           \mu_g\big(Q_1\times Q_2\big) = \widetilde{O}(n^{-8\kappa}) \, , 
        \end{equation}
        so by taking $\kappa>0$ sufficiently small~\eqref{eq:bound_on_gaussian_meausure_on_lattice_cubes} also applies for $\mu$ in place of $\mu_g$. Since $\log_+|\cdot|$ has uniformly bounded variation in $\bC$, an immediate consequence of~\eqref{eq:comparison_of_measures_on_lattice_cubes_complex_coordinates} is that 
        \begin{equation}
        \label{eq:comparison_of_measures_with_logarithmic_integral}
            \iint_{Q_1\times Q_2} \log_+|Y_1| \cdot \log_+|Y_2| \, {\rm d}\mu \\ = \iint_{Q_1\times Q_2} \log_+|Y_1| \cdot \log_+|Y_2| \, {\rm d}\mu_g + \widetilde{O}(n^{-8\kappa})
        \end{equation}
        for any $Q_1\in \overline{\Ga}_z$ and $Q_2\in \overline{\Ga}_w$. We get that
        \begin{align*}
            \iint_{\substack{  (X_1,Y_1)\in \cE_z \\  (X_2,Y_2)\in \cE_w }} & \log_+|Y_1| \cdot \log_+|Y_2| \, {\rm d}\mu(X_1,Y_1,X_2,Y_2) \\ & \le \sum_{Q_1\in \overline{\Ga}_z} \sum_{Q_2\in \overline{\Ga}_z} \iint_{Q_1\times Q_2} \log_+|Y_1| \cdot \log_+|Y_2| \, {\rm d}\mu \\ (\text{by}~\eqref{eq:comparison_of_measures_with_logarithmic_integral}) &\le \sum_{Q_1\in \overline{\Ga}_z} \sum_{Q_2\in \overline{\Ga}_z} \iint_{Q_1\times Q_2} \log_+|Y_1| \cdot \log_+|Y_2| \, {\rm d}\mu_g + \widetilde{O}\big(|\overline{\Ga}_z|^2 n^{-8\kappa}\big) 
            \\ (\text{by}~\eqref{eq:bound_on_caridinality_of_Gamma}) & \le \sum_{Q_1\in \Ga_z} \sum_{Q_2\in \Ga_z} \iint_{Q_1\times Q_2} \log_+|Y_1| \cdot \log_+|Y_2| \, {\rm d}\mu_g + \widetilde{O}\big(n^{-4\beta} + n^{-2\kappa -2\beta} \big) 
            \\ & \le 
            \iint_{\substack{  (X_1,Y_1)\in \cE_z \\  (X_2,Y_2)\in \cE_w }} \log_+|Y_1| \cdot \log_+|Y_2| \, {\rm d}\mu_g(X_1,Y_1,X_2,Y_2)  + \widetilde{O}(n^{-4\beta}) \, , 
        \end{align*}
        which proves the upper bound in~\eqref{eq:proof_of_lemma_comparison_for_C_what_we_want_in_complex_coordinates}, for $\beta>0$ sufficiently small. The lower bound also follows by switching between $\mu$ and $\mu_g$ is the above derivation, and altogether we have proved~\eqref{eq:proof_of_lemma_comparison_for_C_what_we_want_in_complex_coordinates}, which concludes the proof of the lemma.
    \end{proof}
    With Lemma~\ref{lemma:comparison_with_a_gaussian_on_the_event_C_pm}, we give a simple proof which we owe from the previous section. 
    \begin{proof}[Proof of Claim~\ref{claim:sum_of_probabilities_over_net_points_is_bounded}]
        By Lemma~\ref{lemma:comparison_with_a_gaussian_on_the_event_C_pm}, item~\eqref{eq:lemma_comparison_with_a_gaussian_on_the_event_C_pm_first} we have 
        \[
        \bigg|\frac{1}{n} \sum_{z\in \cN} \bP\big(\cC_z^+ \big) - \frac{1}{n} \sum_{z\in \cN} \bP_g\big(\cC_z^+ \big) \bigg|\lesssim n^{-1/20} \, \frac{|\cN|}{n} \lesssim n^{3\beta-1/20} \, .
        \]
        Therefore, it is enough to prove the corresponding bound for the Gaussian measure $\bP_g$. Let $\nu=\nu_z$ denote the law of 
        \[
        \Big(\frac{g_n(z)}{\sqrt{n}},\frac{g_n^\prime(z)}{n^{3/2}}\Big)
        \]
        induced on $\bC^2$. Borrowing the notation from the proof of Lemma~\ref{lemma:comparison_with_a_gaussian_on_the_event_C_pm}, we have
        \[
        \bP_g\big(\cC_z^+\cap \cG\big) = \nu(\cE_z) \, .
        \]
        As $\nu$ has uniformly bounded density we get from~\eqref{eq:bound_on_measure_of_E_z} that
        \[
        \bP_g\big(\cC_z^+\big)  \le \widetilde{O}(n^{-2\beta}) + \bP(\cG^c) \le \widetilde{O}(n^{-2\beta}) \, ,
        \]
        and hence
        \[
        \frac{1}{n} \sum_{z\in \cN} \bP_g\big(\cC_z^+ \big) = \widetilde{O}\Big(n^{-2\beta} \, \frac{|\cN|}{n}\Big) = \widetilde{O}(1) \,. \qedhere 
        \]
    \end{proof}
    \subsection{Application of the Kac-Rice formula}
    \label{subsection:application_of_kac_rice}
    In light of Lemma~\ref{lemma:comparison_with_a_gaussian_on_the_event_C_pm}, we can now turn our attention to compute the relevant expectation with respect to the Gaussian polynomial $g_n$ given by~\eqref{eq:def_of_random_polynomial_with_gaussian_coefficients}. For $t\ge 0$ we set
    \begin{equation}
    \label{eq:def_of_normalized_gaussian_random_polynomial}
        G_n(t) = \frac{1}{\sqrt{n}} \,  g_n(e^{-t/n}) = \frac{1}{\sqrt{n}} \sum_{k=0}^{n} \gamma_k  \, e^{-tk/n} \, ,
    \end{equation}
    and
    \begin{equation}
        \label{eq:def_of_kac_rice_density}
        \psi_n(t) = \frac{\bE_g\Big[ |G_n^\prime(t)|^2 \log|G_n^\prime(t)| \  \big| \  G_n(t)=0\Big]}{ \bE_g\big[|G_n(t)|^2\big]} \, .
    \end{equation}
    \begin{lemma}
        \label{lemma:application_of_Kac_Rice}
        With $\psi_n$ given by~\eqref{eq:def_of_kac_rice_density} and $\displaystyle T_n = -n\log\Big(1-\frac{\log^3 n}{n} \Big) $, we have 
        \[
        \bE\bigg[ \,  \frac{1}{n} \sum_{z \in \cN} \log \Big|\frac{f_n'(z)}{n^{3/2}} \Big| \cdot \one_{\cC_z^+} \bigg] = \int_{0}^{T_n} \psi_n(t) \, {\rm d}t + o(1)
        \]
        as $n\to \infty$.  
    \end{lemma}
    \begin{proof}
        By Lemma~\ref{lemma:comparison_with_a_gaussian_on_the_event_C_pm}, item~\eqref{eq:lemma_comparison_with_a_gaussian_on_the_event_C_pm_second}, we have 
        \[
        \bE\bigg[ \,  \frac{1}{n} \sum_{z \in \cN} \log \Big|\frac{f_n'(z)}{n^{3/2}} \Big| \cdot \one_{\cC_z^+} \bigg] = \bE_g\bigg[ \,  \frac{1}{n} \sum_{z \in \cN} \log \Big|\frac{g_n'(z)}{n^{3/2}} \Big| \cdot \one_{\cC_z^+} \bigg] + O(n^{-1/20+2\beta})
        \]
        which is precisely~\eqref{eq:expectation_for_polynomial_and_gaussian_over_net_is_the_same}. Having proven that, we now `unwind' the Gaussian expectation back to the original sum over roots in $\mathcal{A}\setminus \mathcal{S}$, so that we can apply the Kac-Rice formula. Indeed, combining the above with Lemma~\ref{lemma:sum_over_roots_and_sum_over_net_are_close_in_L_1} (applied to the Gaussian polynomial $g_n$) yields that 
        \[
        \bE\bigg[ \,  \frac{2}{n} \sum_{z \in \cN} \log \Big|\frac{f_n'(z)}{n^{3/2}} \Big| \cdot \one_{\cC_z^+} \bigg] = \bE_g\bigg[\frac{1}{n} \sum_{\alpha\in \mathcal{A}\setminus \mathcal{S}} \log\Big|\frac{g_n^\prime(\alpha)}{n^{3/2}}\Big| \,  \one\Big\{\frac{|g_n^\prime(\alpha)|}{n^{3/2} } \ge \frac{1}{\log^{10}n}\Big\}\bigg] + o(1) \, ,
        \]
        where we recall that $\cA$ is given by~\eqref{eq:def_of_annulus_A} and $\mathcal{S}$ is given by~\eqref{eq:def_of_almost_real_R}. Since $g_n$ is a Gaussian process which is (almost surely) analytic, a sufficiently general Kac-Rice formula, see for instance~\cite[Theorem~6.4]{Azais-Wschebor} (see also~\cite[Appendix~A]{Michelen-Yakir-log-energy} for a related derivation) yields
        \begin{multline}
            \label{eq:gaussian_expectation_after_application_of_kac_rice}
            \bE_g\bigg[\frac{1}{n} \sum_{\alpha\in \mathcal{A}\setminus \mathcal{S}} \log\Big|\frac{g_n^\prime(\alpha)}{n^{3/2}}\Big| \,  \one\Big\{\frac{|g_n^\prime(\alpha)|}{n^{3/2} } \ge \frac{1}{\log^{10}n}\Big\}\bigg] \\  = \frac{1}{n}\int_{\mathcal{A}\setminus \mathcal{S}} \, \frac{\bE_g\Big[ |g_n^\prime(z)|^2 \log\Big|\frac{g_n^\prime(z)}{n^{3/2}}\Big| \cdot \one\big\{|g_n^\prime(z)|\ge\frac{n^{3/2}}{\log^{10}n} \big\} \ \mid g_n(z)=0\Big]}{\pi \bE_g\big[|g_n(z)|^2\big]}  \,  {\rm d}m(z) \, .
        \end{multline}
        Our goal now is to simplify the integral on the right-hand side of~\eqref{eq:gaussian_expectation_after_application_of_kac_rice}. First, we note that $g_n$ is rotation invariant (since the complex Gaussian distribution is such), so this integral is in fact equal to
        \[ 
        \frac{2}{n}\int_{1-\log^3n/n}^1 \, \frac{\bE_g\Big[ |g_n^\prime(r)|^2 \log\Big|\frac{g_n^\prime(r)}{n^{3/2}}\Big| \cdot \one\big\{|g_n^\prime(r)|\ge\frac{n^{3/2}}{\log^{10}n} \big\} \ \mid g_n(r)=0\Big]}{ \bE_g\big[|g_n(r)|^2\big]} \, r \,  {\rm d}r \, .
        \]
        By applying the change of variables $r=e^{-t/n}$ and plugging in the definition~\eqref{eq:def_of_normalized_gaussian_random_polynomial} of $G_n$, the above integral is equal to 
        \[
        2 \int_{0}^{T_n}\frac{\bE_g\Big[ |G_n^\prime(t)|^2 \log|G_n^\prime(t)| \cdot \one\big\{|G_n^\prime(t)|\ge \frac{1}{\log^{10}n}\big\} \  \big| \  G_n(t)=0\Big]}{ \bE_g\big[|G_n(t)|^2\big]} \, e^{-2t/n} \, {\rm d}t \, . 
        \]
        Plugging this into~\eqref{eq:gaussian_expectation_after_application_of_kac_rice} yields that
        \begin{multline*}
                \bE\bigg[ \,  \frac{1}{n} \sum_{z \in \cN} \log \Big|\frac{f_n'(z)}{n^{3/2}} \Big| \cdot \one_{\cC_z^+} \bigg]         
              \\ =  \int_{0}^{T_n}\frac{\bE_g\Big[ |G_n^\prime(t)|^2 \log|G_n^\prime(t)| \cdot \one\big\{|G_n^\prime(t)|\ge \frac{1}{\log^{10}n}\big\} \  \big| \  G_n(t)=0\Big]}{ \bE_g\big[|G_n(t)|^2\big]} \, e^{-2t/n} \, {\rm d}t + o(1) \, . 
        \end{multline*}
        To conclude the proof, it remains to note that $\psi_n$ given by~\eqref{eq:def_of_kac_rice_density} is uniformly bounded in $t\in[0,T_n]$, since the numerator is uniformly bounded from below and the denominator is uniformly bounded from above. In particular, we have 
        \[
        \frac{\Big|\bE_g\Big[ |G_n^\prime(t)|^2 \log|G_n^\prime(t)| \cdot \one\big\{|G_n^\prime(t)| < \frac{1}{\log^{10}n}\big\} \  \big| \  G_n(t)=0\Big]\Big|}{ \bE_g\big[|G_n(t)|^2\big]}  \lesssim \frac{\log \log n}{\log ^{20}n} \, ,
        \]
        and hence
        \[
        \int_{0}^{T_n}   \frac{\Big|\bE_g\Big[ |G_n^\prime(t)|^2 \log|G_n^\prime(t)| \cdot \one\big\{|G_n^\prime(t)| < \frac{1}{\log^{10}n}\big\} \  \big| \  G_n(t)=0\Big] \Big|}{ \bE_g\big[|G_n(t)|^2\big]}  \, {\rm d}t \lesssim \frac{\log^3 n}{\log^{19}n}  = o(1) \, .
        \]
        Since we also have $\displaystyle \int_0^{T_n} |e^{-2t/n}-1| \,  {\rm d}t = o(1)$, we arrive at
        \begin{equation*}
                \bE\bigg[ \,  \frac{1}{n} \sum_{z \in \cN} \log \Big|\frac{f_n'(z)}{n^{3/2}} \Big| \cdot \one_{\cC_z^+} \bigg]         
              \\ =  \int_{0}^{T_n}\frac{\bE_g\Big[ |G_n^\prime(t)|^2 \log|G_n^\prime(t)|\  \big| \  G_n(t)=0\Big]}{ \bE_g\big[|G_n(t)|^2\big]}  \, {\rm d}t + o(1) \, . 
        \end{equation*}
        and we are done.
    \end{proof}
    \noindent
    It remains to compute the leading order in the Gaussian integral from Lemma~\ref{lemma:application_of_Kac_Rice}. We document this simple computation in the next claim, which will be proved in the next sub-section. 
    \begin{claim}
        \label{claim:computing_the_limit_for_gaussian_integral}
        Let $\psi_n(t)$ be given by~\eqref{eq:def_of_kac_rice_density}, then for every $t\ge 0$ we have 
        \[
        \lim_{n\to \infty} \psi_n(t) = \Psi(t) \, ,
        \]
        where
        \begin{equation}
        \label{eq:def_of_Psi}
            \Psi (t) = \Big(\frac{1}{t^2} - \frac{1}{\sinh^2(t)}\Big) \cdot \Big(\frac{\log S(t) + 1-\gamma}{2}\Big) \, ,
        \end{equation}
        and
        \[
        S(t) = \frac{\big(1+2t^2-\cosh(2t)\big)\cdot \big( 1-\coth(t)\big)}{2t^3} \, .
        \]
        Furthermore, we have
        \[
        \lim_{n\to \infty} \int_{0}^{T_n} \psi_n(t) \,  {\rm d}t = \int_{0}^{\infty}\Psi(t) \, {\rm d}t  \, .
        \]
    \end{claim}
    \noindent We note by passing that the pre-factor $$\displaystyle \frac{1}{t^2} - \frac{1}{\sinh^2(t)}$$ in the limiting density $\Psi$ is, up to normalization, the limiting radial density for the expected number of roots of random Kac polynomials in the disk, see~\cite[Theorem~2]{Ibragimov-Zeitouni}. As we shall see below, Claim~\ref{claim:computing_the_limit_for_gaussian_integral} follows from a simple Gaussian computation of the density~\eqref{eq:def_of_kac_rice_density}. Before we perform this computation we show how Lemma~\ref{lemma:computation_of_the_mean_for_sum_over_net} is derived, and in particular pin down the relation between $\Psi$ defined in~\eqref{eq:def_of_Psi} and $\Phi$ from Theorem~\ref{thm:LLN_for_L_n}.

    \begin{proof}[Proof of Lemma~\ref{lemma:computation_of_the_mean_for_sum_over_net}]
        By combining Lemma~\ref{lemma:application_of_Kac_Rice} with Claim~\ref{claim:computing_the_limit_for_gaussian_integral}, we see that 
        \[
        \lim_{n\to\infty} \bE\bigg[ \,  \frac{2}{n} \sum_{z \in \cN} \log \Big|\frac{f_n'(z)}{n^{3/2}} \Big| \cdot \one_{\cC_z^+} \bigg] = 2\int_0^\infty \Psi(t) \, {\rm d}t
        \]
        with $\Psi$ given by~\eqref{eq:def_of_Psi}. Observe that
        \[
        \int_{0}^\infty \Big(\frac{1}{t^2} - \frac{1}{\sinh^2(t)}\Big) \, {\rm d}t = \Big[-\frac{1}{t} + \coth(t) \Big]_{t=0}^{t=\infty} = 1 \, ,
        \]
        which implies that
        \begin{align*}
            2\int_{0}^{\infty} \Psi(t) \, {\rm d}t &= (1-\gamma) \int_{0}^\infty \Big(\frac{1}{t^2} - \frac{1}{\sinh^2(t)}\Big) \, {\rm d}t + \int_{0}^\infty \Big(\frac{1}{t^2} - \frac{1}{\sinh^2(t)}\Big) \log S(t) \, {\rm d}t \\ &= 1-\gamma + \int_{0}^\infty \Phi(t) \, {\rm d}t 
        \end{align*}
        where
        \[
        \Phi(t) = \Big(\frac{1}{t^2} - \frac{1}{\sinh^2(t)}\Big) \cdot \log \bigg( \frac{\big(1+2t^2-\cosh(2t)\big)\cdot \big( 1-\coth(t)\big)}{2t^3} \bigg) \, .
        \]
        This limit matches the definition~\eqref{eq:def_of_c_ast} of ${\sf c}_\ast$, and hence we are done. 
    \end{proof}
    
    \subsection{Computing the limiting integral}
    Recall the definition~\eqref{eq:def_of_normalized_gaussian_random_polynomial} for the normalized Gaussian random polynomial $G_n$. Our goal here is to prove Claim~\ref{claim:computing_the_limit_for_gaussian_integral}, that is compute the limit of the density $\psi_n$ given by~\eqref{eq:def_of_kac_rice_density} as $n\to\infty$ and also justify that the corresponding integral converges. Denote by
    \begin{equation}
    \label{eq:def_of_s_n}
        s_n(t) = \bE_g[|G_n(t)|^2 ] = \frac{1}{n} \sum_{k=0}^n e^{-2tk/n} \, .
    \end{equation}
    \begin{claim}
        \label{claim:convergence_of_s_n}
        For all $t\ge 0$ and $\ell\in\{0,1,2\}$ we have
        \begin{equation}
        \label{eq:claim:convergence_of_s_n_first}
        \lim_{n\to\infty} s_n^{(\ell)}(t) = (-2)^{\ell} \int_{0}^1 \la^\ell e^{-2t\la } \, {\rm d}\la = e^{-2t}\times \begin{cases}
            \frac{e^{2t}-1}{2t} & \ell = 0 \, , \\ \frac{1+2t-e^{2t}}{2t^2} & \ell =1 \, , \\ \frac{e^{2t}-1-2t-2t^2}{t^3} & \ell = 2\, . 
        \end{cases}    
        \end{equation}
        Furthermore, we have
        \begin{equation}
        \label{eq:claim:convergence_of_s_n_second}
        \sup_{\ell\in\{0,1,2\}} \, \sup_{n\ge 2} \,  \sup_{t\in[0,M]} |s_n^{(\ell)}(t)| = O_M(1) \, .    
        \end{equation}
        Lastly, for all $t\in[M,\log^4 n]$ for with $M$ being sufficiently large but fixed we have
        \begin{equation}
        \label{eq:claim:convergence_of_s_n_third}
          |s_n^{(\ell)}(t)| \in \Big(\frac{0.1}{ t^{\ell + 1}},\frac{10}{t^{\ell +1}}\Big)  
        \end{equation}
        for all $n$ large enough. 
    \end{claim}
    \begin{proof}
    Note that
    \[
    s_n^{(\ell)}(t) = \frac{1}{n^{\ell + 1}}   \sum_{k=0}^{n} (-2k)^{\ell} e^{-2tk/n}
    \]
    which is a Riemann sum for the limiting integral~\eqref{eq:claim:convergence_of_s_n_first}, with $n+1$ equally spaced points in $[0,1]$. Hence, the Euler-Maclaurin summation formula implies that 
    \[
    s_n^{(\ell)}(t) = (-2)^{\ell} \int_{0}^1 \la^\ell e^{-2t\la } \, {\rm d}\la + \widetilde{O}(n^{-1}) 
    \]
    uniformly for $t\in[0,\log^4n]$, from which both~\eqref{eq:claim:convergence_of_s_n_first} and~\eqref{eq:claim:convergence_of_s_n_second} follow immediately. To prove~\eqref{eq:claim:convergence_of_s_n_third}, we only need to observe that 
    \[
    \lim_{t\to \infty} t^{1+\ell}  \int_{0}^1 \la^\ell e^{-2t\la } \, {\rm d}\la = \lim_{t\to\infty} \int_0^t x^\ell e^{-2x} \, {\rm d}x \in \Big[\frac14,\frac12\Big] \, . \qedhere
    \]
    \end{proof}
    \begin{claim}
    \label{claim:compute_numerator_in_kac_rice_density}
        Let $Z$ be a mean-zero complex Gaussian with $\bE[|Z|^2] = s$. Then
        \[
        \bE\big[|Z|^2 \log|Z|\big] = \frac{s}{2}\big(\log s + 1-\gamma\big)
        \]
        where $\gamma$ is Euler's constant. 
    \end{claim}
    \begin{proof}
        Let $\widehat Z = Z/\sqrt{s}$ and note that $\widehat{Z}$ is a standard complex Gaussian. We have
        \[
        \bE\big[|Z|^2 \log|Z|^2\big] = s\Big( \bE\big[|\widehat Z|^2 \log|\widehat Z|^2\big] + \bE[|\widehat Z|^2] \cdot \log s \Big) = s\Big( \bE\big[|\widehat Z|^2 \log|\widehat Z|^2\big] +  \log s \Big)  \, .
        \]
        It remains to note that 
        \[
        \bE\big[|\widehat Z|^2 \log|\widehat Z|^2\big] = \int_{0}^\infty (\log x) \, x \, e^{-x} \, {\rm d}x = 1-\gamma\, . \qedhere
        \]
    \end{proof}
    \noindent
    Combining the above claims, we are ready to conclude the computation for the limiting constant. 
    \begin{proof}[Proof of Claim~\ref{claim:computing_the_limit_for_gaussian_integral}]
        Recall that $\psi_n$ is given by~\eqref{eq:def_of_kac_rice_density}. Its denominator is equal to $s_n$, given by~\eqref{eq:def_of_s_n}. We can compute the numerator in $\psi_n$ as well. Indeed, the conditional of $G_n^\prime(t)$ given that $G_n(t) = 0$ is a mean-zero complex Gaussian with variance given by the Schur complement, i.e.
        \begin{equation*}
            \widetilde{s}_n(t) = \bE\big[|G_n^\prime(t)|^2\big] - \frac{\Big|\bE\big[G_n(t) \overline{G_n^\prime(t)}\big]\Big|^2}{\bE\big[|G_n(t)|^2\big]} = s_n^{\prime\prime}(t) - \frac{(s_n^\prime(t))^2}{s_n(t)} \, .
        \end{equation*}
        Therefore, by Claim~\ref{claim:compute_numerator_in_kac_rice_density} we have 
        \[
        \bE_g\Big[ |G_n^\prime(t)|^2 \log|G_n^\prime(t)| \  \big| \  G_n(t)=0\Big] = \frac{\widetilde{s}_n(t)}{2}\Big(\log \widetilde{s}_n(t) + 1 - \gamma \Big) \, ,
        \]
        and altogether
        \begin{equation}
        \label{eq:psi_n_after_computation}
            \psi_n(t) = \widetilde{s}_n(t)\, \frac{\log \widetilde{s}_n(t) + 1 - \gamma}{2s_n(t)} \, .
        \end{equation}
        We are now ready to take the limit $n\to \infty$. Claim~\ref{claim:convergence_of_s_n} and some algebra gives that
        \begin{align}
            \label{eq:limit_for_s_n_tilde}
            \nonumber
            \lim_{n\to\infty} \widetilde{s}_n(t) &= 4\int_0^1 \la^2 e^{-2t\la} \, {\rm d}\la - 4 \, \frac{\Big(\int_0^1 \la e^{-2t\la} \, {\rm d}\la\Big)^2}{\int_0^1 e^{-2t\la} \, {\rm d}\la} \\ \nonumber & = \frac{8e^{-4t}}{e^{2t}-1}\bigg( \frac{(e^{2t}-
            1)(e^{2t}-1-2t-2t^2)}{2t^3}  -  \frac{(1+2t-e^{2t})^2}{4t^3}\bigg) \\ & = \frac{\big(1+2t^2-\cosh(2t)\big)\cdot \big( 1-\coth(t)\big)}{2t^3} = S(t) \, .
        \end{align}
        Furthermore, we have 
        \begin{equation*}
        \lim_{n\to\infty} \frac{\widetilde{s}_n(t)}{s_n(t)} = \frac{2t}{1-e^{-2t}} \, S(t) = \frac{1}{t^2} - \frac{1}{\sinh^2(t)} \, .
        \end{equation*}
         Plugging~\eqref{eq:limit_for_s_n_tilde} and the above into~\eqref{eq:psi_n_after_computation}, we finally arrive at
        \begin{equation}
        \label{eq:Psi_is_the_limit_of_psi_n}
            \Psi(t) = \lim_{n\to\infty} \psi_n(t) = \Big(\frac{1}{t^2} - \frac{1}{\sinh^2(t)}\Big) \cdot \Big(\frac{\log S(t) + 1-\gamma}{2}\Big) \, ,
        \end{equation}
        which is what we wanted to show. To justify the limit of integrals, we note that Claim~\ref{claim:convergence_of_s_n} together with~\eqref{eq:psi_n_after_computation} shows that
        \begin{equation*}
            |\psi_n(t)| \lesssim \frac{|\log t|}{1+t^2} \, ,
        \end{equation*}
        uniformly as $n\to \infty$. Thus, we can apply the Lebesgue dominated convergence theorem together with~\eqref{eq:Psi_is_the_limit_of_psi_n} and get that
        \[
         \lim_{n\to \infty} \int_{0}^{T_n} \psi_n(t) \,  {\rm d}t = \int_{0}^\infty \lim_{n\to\infty} \psi_n(t) \, {\rm d}t = \int_{0}^{\infty}\Psi(t) \, {\rm d}t \, ,
        \]
        as desired. 
    \end{proof}
    \section{Concentration of the Mahler measure}
    \label{sec:concentration_of_the_Mahler_measure}
    \noindent
    In this section we prove Proposition~\ref{prop:LLN_for_log_mahler_measure}, which states that
    \begin{equation}
    \label{eq:concentration_of_mahler_measure_what_we_want}
        \int_0^1\log \left|\frac{f_n(e^{2\pi \I \theta})}{\sqrt{n}} \right|\,{\rm d}\theta \xrightarrow[n\to \infty]{\bP} -\frac{\gamma}{2} \, .
    \end{equation}
    To lighten on the notation, throughout the section we write
    \begin{equation}
        \label{eq:polynomial_normalized_on_the_circle}
        \widetilde{f}_n(\theta) = \frac{|f_n(e^{\I\theta})|}{\sqrt{n}} \, .
    \end{equation}
    We also recall the Gaussian random polynomial $g_n$ given by~\eqref{eq:def_of_random_polynomial_with_gaussian_coefficients}, and define $\widetilde{g}_n$ analogously to~\eqref{eq:polynomial_normalized_on_the_circle}. 
    
    \subsection{Gaussian comparison revisited}
    To prove concentration for the Mahler measure, we will basically show that the distribution of the random polynomial $f_n$ on the unit circle is close to its Gaussian counterpart $g_n$. Throughout this section, we denote by 
    \begin{equation}
        \label{eq:def_of_I}
        \mathcal{I}  = [0,n^{-1/2}]\cup[1/2-n^{-1/2},1/2+n^{-1/2}]\cup[1-n^{-1/2},1] 
    \end{equation}
    and note that the measure of $\mathcal{I}$ is $O(n^{-1/2})$. 
    \begin{claim}
        \label{claim:gaussian_comparison_CDF_on_unit_circle}
        For all $\theta \in [0,1]\setminus \mathcal{I}$ we have
        \[
        \sup_{x\in \bR} \Big| \, \bP\big( \, \widetilde{f}_n(\theta)^2 \le x \, \big) - \bP_g\big( \, \widetilde{g}_n(\theta)^2 \le x \, \big) \Big| \le \widetilde{O}(n^{-1/2}) \, .
        \]
        Furthermore, for all $\theta,\varphi \in  [0,1]\setminus \mathcal{I}$ such that $|\theta-\varphi| \ge n^{-1/2}$ we have
        \[
        \sup_{(x,y)\in \bR^2} \Big| \, \bP\big( \, \widetilde{f}_n(\theta)^2 \le x \, , \,  \widetilde{f}_n(\varphi)^2 \le y \, \big) - \bP_g\big( \, \widetilde{g}_n(\theta)^2 \le x \, \big) \cdot \bP_g\big( \, \widetilde{g}_n(\varphi)^2 \le y \, \big) \Big| \le \widetilde{O}(n^{-1/2}) \, .
        \]
    \end{claim}
    \noindent
    We note that by the rotation invariance of the Gaussian polynomial $g_n$, we in fact have
    \begin{equation*}
        \bP_g\big( \, \widetilde{g}_n(\theta)^2 \le x \, \big) = \bP(|Z|^2 \le x) = 1-e^{-x} \, ,
    \end{equation*}
    for all $\theta\in [0,1]$. Here $Z$ is a standard complex Gaussian. We shall apply this observation later on, when computing the limiting expectation for the integral~\eqref{eq:concentration_of_mahler_measure_what_we_want}, but for now the current formulation of Claim~\ref{claim:gaussian_comparison_CDF_on_unit_circle} is more convenient. 
    \begin{proof}[Proof of Claim~\ref{claim:gaussian_comparison_CDF_on_unit_circle}]
        This follows immediately from Lemma~\ref{lemma:berry_esseen_in_annulus}, by restricting the supremum over convex sets $K$ to balls centered at the origin. 
    \end{proof}

    \subsection{Proving concentration}
    We denote by
    \begin{equation}
        \label{eq:def_of_event_B}
        \mathcal{B} = \Big\{ \min_{\theta\in[0,1]} \widetilde{f}(\theta) \le n^{-3/2}\Big\}
    \end{equation}
    By Claim~\ref{claim:no_root_on_unit_circle} (which, as we recall, is an immediate corollary of the main result in~\cite{Cook-Nguyen}), we have $$ \lim_{n\to\infty} \bP(\cB) =0 \, .$$ In fact, the corresponding Gaussian fact $\displaystyle \lim_{n\to\infty} \bP_g(\cB) =0 $ was already derived in~\cite{Konyagin-Schlag,Yakir-Zeitouni}. With that, the proof of Proposition~\ref{prop:LLN_for_log_mahler_measure} breaks down into three simple steps. 
    \begin{claim}
        \label{claim:computation_of_expected_gaussian_mahler_measure}
        For the complex Gaussian random polynomial $g_n$ given by~\eqref{eq:def_of_random_polynomial_with_gaussian_coefficients} we have
        \[
        \bE_g \Big[ \int_{0}^1 \big| \log \widetilde g_n(\theta) \big|  \, {\rm d} \theta \Big] < \infty \, ,
        \]
        and furthermore, for all $n\ge 1$,
        \[
         \bE_g \Big[ \int_{0}^1  \log \widetilde g_n(\theta)   \, {\rm d} \theta \Big] = -\frac{\gamma}{2} \, .
        \]
    \end{claim}
    
    \begin{proof}
        The claim follows trivially once we note that for all $\theta\in[0,1]$ the random variable
        \[
        \frac{1}{\sqrt{n}} g_n(e^{2\pi\I\theta}) = \frac{1}{\sqrt{n}} \sum_{k=0}^{n} \gamma_k e^{2\pi \I k\theta}
        \]
        has the standard complex Gaussian law. Then it is immediate to check that $\bE[\big|\log|Z|\big|] < \infty$ and also
        \[
        \bE[|\log|Z|]= \frac{1}{\pi} \int_{\bC} \log|z| \, e^{-|z|^2} \, {\rm d}m(z)  = \frac{1}{2}\int_0^\infty (\log s)\, e^{-s} \, {\rm d}s = -\frac{\gamma}{2} \, ,
        \]
        where $\gamma$ is Euler's constant. 
    \end{proof}

    \begin{lemma}
        \label{lemma:mahler_measure_of_gaussian_and_original_are_close}
        We have
        \[
        \lim_{n\to\infty} \bigg( \bE\Big[ \int_{0}^1 \log \widetilde{f}_n(\theta) \cdot  \one\big\{ \widetilde f_n(\theta) \ge n^{-3/2}\big\} \, {\rm d}\theta \Big] -  \bE_g\Big[ \Big(\int_{0}^1 \log \widetilde{g}_n(\theta) \, {\rm d}\theta\Big) \Big] \bigg) = 0 \, .
        \]
    \end{lemma}
    \begin{proof}
        We first note that for each $n\ge 2$
        \[
        \sup_{\theta\in[0,1]} \, \bE\Big[\big|\log \widetilde{f}_n(\theta) \big| \, \one\big\{ \widetilde f_n(\theta) \ge n^{-3/2}\big\} \Big] < \infty \, .
        \]
        Therefore, we can apply Fubini's theorem and observe that, for all fixed $M\ge 1$, 
        \begin{align}
        \label{eq:computing_the_expected_mahler_measure_fubini}
            \nonumber \bE\Big[ \int_{0}^1 \log \widetilde{f}_n(\theta) \cdot   \one\big\{  n^{-3/2} \le \widetilde f_n(\theta) \le M \big\} \, {\rm d}\theta \Big]  &= \int_{0}^{1} \bE\big[\log \widetilde{f}_n(\theta) \cdot \one\big\{  n^{-3/2} \le \widetilde f_n(\theta) \le M \big\} \big] \, {\rm d}\theta \\  &= \int_{0}^{1} \int_{n^{-3/2}}^M \log x  \, {\rm d} \big(\bP(\widetilde f_n(\theta) \le x) \big) \, {\rm d}\theta
        \end{align}
        By the Esseen bound on small-ball probability of non-degenerate random variables~\cite{Esseen} (see also \cite[Lemma~2.4]{Michelen-Yakir-root-separation}), we have for $x\ge n^{-1/2}$,
        \begin{equation}
            \label{eq:eseeen_inequality}
            \bP(\widetilde f_n(\theta) \le x) \lesssim x  
        \end{equation}
        uniformly for $\theta\in[0,1]$. Hence, we have
        \[
        \bigg| \int_{0}^{1} \int_{n^{-3/2}}^{n^{-1/2}} \log x  \, {\rm d} \big(\bP(\widetilde f_n(\theta) \le x) \big) \, {\rm d}\theta \bigg| \lesssim \log n  \cdot  \bP(\widetilde f_n(\theta) \le n^{-1/2}) \lesssim \frac{\log n}{\sqrt{n}} \, , 
        \]
        and~\eqref{eq:computing_the_expected_mahler_measure_fubini} yields that
        \begin{align*}
        \bE\Big[ \int_{0}^1 \log \widetilde{f}_n(\theta) \cdot  & \one\big\{  n^{-3/2} \le \widetilde f_n(\theta) \le M \big\} \, {\rm d}\theta \Big]  \\ & =
            \int_{0}^1 \bigg(\Big[\log x \cdot\bP(\widetilde f_n(\theta) \le x) \Big]_{x=n^{-1/2}}^{x=M} - \int_{n^{-1/2}}^{M} \frac{\bP(\widetilde f_n(\theta) \le x)}{x} \, {\rm d} x\bigg) \,  {\rm d}\theta  + o(1) \\ &= \log M\int_{0}^{1} \bP(\widetilde f_n(\theta) \le M) \,  {\rm d}\theta - \int_{0}^{1} \int_{n^{-1/2}}^{M} \frac{\bP(\widetilde f_n(\theta) \le x)}{x} \, {\rm d} x \,  {\rm d}\theta + o(1) \, .
        \end{align*}
        By repeating the same computation for the Gaussian polynomial, we see that 
        \begin{align}
            \label{eq:computation_of_mahler_measure_expectation_comparing_with_gaussian}
            \nonumber
           \bE\Big[ \int_{0}^1 \log \widetilde{f}_n(\theta) \cdot &  \one\big\{  n^{-3/2} \le \widetilde f_n(\theta) \le M \big\} \, {\rm d}\theta \Big] - \bE_g\Big[ \int_{0}^1 \log \widetilde{g}_n(\theta) \cdot   \one\big\{  n^{-3/2} \le \widetilde g_n(\theta) \le M \big\} \, {\rm d}\theta \Big] \\ &= o(1) + \log M \int_0^1 \Big(\bP(\widetilde f_n(\theta) \le M) - \bP_g(\widetilde g_n(\theta) \le M)\Big) \, {\rm d}\theta \\ \nonumber & \qquad \quad \qquad - \int_0^1 \int_{n^{-1/2}}^M \frac{\bP(\widetilde f_n(\theta) \le x) - \bP_g(\widetilde g_n(\theta) \le x)}{x} \, {\rm d} x \,  {\rm d}\theta  \, ,
        \end{align}
        so we want to show that the right-hand side of~\eqref{eq:computation_of_mahler_measure_expectation_comparing_with_gaussian} tends to zero as $n\to\infty$. By Claim~\ref{claim:gaussian_comparison_CDF_on_unit_circle} we have
        \[
         \int_{[0,1]\setminus \mathcal{I}}\int_{n^{-1/2}}^M \frac{\big|\bP(\widetilde f_n(\theta) \le x) - \bP_g(\widetilde g_n(\theta) \le x)\big|}{x} \, {\rm d} x  \, {\rm d}\theta \le \widetilde{O} \Big(n^{-1/2} \int_{n^{-1/2}}^M \frac{{\rm d} x}{x}\Big) \le \widetilde{O}(n^{-1/2}) \, ,
        \]
        and by recalling that $\mathcal{I}$ is given by~\eqref{eq:def_of_I} we also have
        \[
        \int_{\mathcal{I}}\int_{n^{-1/2}}^M \frac{\big|\bP(\widetilde f_n(\theta) \le x) - \bP_g(\widetilde g_n(\theta) \le x)\big|}{x} \, {\rm d} x \, {\rm d}\theta  \lesssim \log n \int_{\mathcal{I}} {\rm d}\theta \lesssim \frac{\log n}{\sqrt{n}} \, .
        \]
        That is, the second integral on the right-hand side of~\eqref{eq:computation_of_mahler_measure_expectation_comparing_with_gaussian} tends to zero. To deal with the first integral, observe that
        \begin{align*}
            \big|\bP(\widetilde f_n(\theta) \le M) - \bP_g(\widetilde g_n(\theta) \le M) \big| = \big|\bP(\widetilde f_n(\theta) > M) - \bP_g(\widetilde g_n(\theta) > M)  \big| \lesssim e^{-cM^2} \, ,
        \end{align*}
        where the inequality is due to the sub-Gaussian assumption on the coefficients of $f_n$. Altogether, we plug into~\eqref{eq:computation_of_mahler_measure_expectation_comparing_with_gaussian} and get that
       \begin{multline*}
           \bigg| \bE\Big[ \int_{0}^1 \log \widetilde{f}_n(\theta) \cdot \one\big\{  n^{-3/2} \le \widetilde f_n(\theta) \le M \big\} \, {\rm d}\theta \Big] - \bE_g\Big[ \int_{0}^1 \log \widetilde{g}_n(\theta) \cdot   \one\big\{  n^{-3/2} \le \widetilde g_n(\theta) \le M \big\} \, {\rm d}\theta \Big] \bigg| \\  \lesssim o(1) + e^{-cM^2} \, .
       \end{multline*} 
       Furthermore, an application of the Cauchy-Schwarz inequality shows that for all $\theta\in[0,1]$
       \[
       \bE\big[|\log \widetilde f_n(\theta)| \cdot \one \{\widetilde f_n(\theta)\ge M \}\big] \lesssim \big(\bP(\widetilde f_n(\theta)\ge M)\big)^{1/2} \lesssim e^{-cM^2} \, ,
       \]
       and hence we get that
       \begin{equation*}
           \limsup_{n\to\infty} \bigg| \bE\Big[ \int_{0}^1 \log \widetilde{f}_n(\theta) \cdot  \one\big\{ \widetilde f_n(\theta) \ge n^{-3/2}\big\} \, {\rm d}\theta \Big] - \bE_g\Big[ \int_{0}^1 \log \widetilde{g}_n(\theta) \cdot  \one\big\{ \widetilde g_n(\theta) \ge n^{-3/2}\big\} \, {\rm d}\theta \Big] \bigg| \lesssim e^{-M^2}
       \end{equation*}
       for all $M\ge 1$. This proves that 
       \[
       \lim_{n\to \infty} \bigg(\bE\Big[ \int_{0}^1 \log \widetilde{f}_n(\theta) \cdot  \one\big\{ \widetilde f_n(\theta) \ge n^{-3/2}\big\} \, {\rm d}\theta \Big] - \bE_g\Big[ \int_{0}^1 \log \widetilde{g}_n(\theta) \cdot  \one\big\{ \widetilde g_n(\theta) \ge n^{-3/2}\big\} \, {\rm d}\theta \Big] \bigg) = 0 \, .
       \]
       To conclude the lemma, it remains to observe that 
       \begin{multline*}
        \bigg|  \bE_g\Big[ \int_{0}^1 \log \widetilde{g}_n(\theta) \cdot  \one\big\{ \widetilde g_n(\theta) \ge n^{-3/2}\big\} \, {\rm d}\theta \Big]-\bE_g\Big[ \int_{0}^1 \log \widetilde{g}_n(\theta)\, {\rm d}\theta \Big]  \bigg|  \\ = \bigg|\bE_g\Big[ \one_{\cB} \cdot \int_{0}^1 \log \widetilde{g}_n(\theta)\, {\rm d}\theta \Big]\bigg|   \lesssim \bP(\cB)^{1/2}  \, ,
       \end{multline*}
       where the inequality follows from Claim~\ref{claim:computation_of_expected_gaussian_mahler_measure}, combined with the Cauchy-Schwarz inequality. By Claim~\ref{claim:no_root_on_unit_circle} the right-hand side of the above display tends to zero as $n\to\infty$, and the lemma follows.  
    \end{proof}
    \begin{lemma}
        \label{lemma:mahler_measure_variance_bound}
        We have
        \[
        \lim_{n\to \infty} \text{\normalfont Var}\Big( \int_{0}^1 \log \widetilde{f}_n(\theta) \cdot  \one\big\{ \widetilde f_n(\theta) \ge n^{-3/2}\big\} \, {\rm d}\theta \, \Big) = 0 \, . 
        \]
    \end{lemma}
    \begin{proof}
        We start by showing that regions where the Gaussian comparison is invalid asymptotically do not contribute to the variance. Indeed,~\eqref{eq:eseeen_inequality} shows that
        \begin{equation}
        \label{eq:variance_of_mahler_measure_reduction_1}
            \bE\bigg|\int_{0}^{1} \log \widetilde{f}_n(\theta) \cdot  \one\big\{ n^{-3/2} \le  \widetilde f_n(\theta) \le n^{-1/2}\big\} \, {\rm d}\theta \,  \bigg|^2 \lesssim (\log n)^2 \Big(\int_{0}^1 \bP\big(\widetilde f_n(\theta) \le n^{-1/2}\big) \,  {\rm d}\theta \Big)^2 \lesssim \frac{\log^2 n}{n} \, .
        \end{equation}
        Furthermore, for all $M\ge 1$ large but fixed an application of Cauchy-Schwarz shows that 
        \begin{equation}
            \label{eq:variance_of_mahler_measure_reduction_2}
            \bE\bigg|\int_{0}^{1}\log \widetilde{f}_n(\theta) \cdot  \one\big\{ \widetilde f_n(\theta) \ge M\big\} \, {\rm d}\theta \,  \bigg|^2 \lesssim \int_0^1 \bE\Big[\log^2 \widetilde f_n(\theta) \cdot \one\big\{\widetilde f_n(\theta) \ge M \big\}\Big] \, {\rm d}\theta \le e^{-cM^2} \, ,
        \end{equation}
        where the last inequality follows from the sub-Gaussian assumption on the coefficients of $f_n$. Recall that $\mathcal{I}$ is given by~\eqref{eq:def_of_I} and has measure $O(n^{-1/2})$. Then clearly
        \begin{equation}
            \label{eq:variance_of_mahler_measure_reduction_3}
            \bigg|\int_{\mathcal{I}} \log \widetilde{f}_n(\theta) \cdot  \one\big\{ n^{-1/2} \le  \widetilde f_n(\theta) \le M\big\} \, {\rm d}\theta \,  \bigg| \lesssim \frac{\log n}{\sqrt{n}} \, .
        \end{equation}
        By combining~\eqref{eq:variance_of_mahler_measure_reduction_1}, \eqref{eq:variance_of_mahler_measure_reduction_2} and~\eqref{eq:variance_of_mahler_measure_reduction_3} we see that
        \begin{multline}
        \label{eq:variance_mahler_measure_bound_after_reductions}
            \limsup_{n\to \infty} \text{\normalfont Var}\Big( \int_{0}^1 \log \widetilde{f}_n(\theta) \cdot  \one\big\{ \widetilde f_n(\theta) \ge n^{-3/2}\big\} \, {\rm d}\theta \, \Big)  \\ \le  \limsup_{n\to \infty} \text{\normalfont Var}\Big( \int_{[0,1]\setminus \mathcal{I}} \log \widetilde{f}_n(\theta) \cdot  \one\big\{ n^{-1/2} \le  \widetilde f_n(\theta) \le M\big\} \, {\rm d}\theta \, \Big) + e^{-cM^2} \, ,
        \end{multline}
        for all $M\ge 1$. To conclude the lemma, it remains to show that the variance on the right-hand side of~\eqref{eq:variance_mahler_measure_bound_after_reductions} tends to zero as $n\to\infty$, for all fixed $M\ge 1$. For that, let us denote 
        \[
        X_n(\theta) = \log \widetilde{f}_n(\theta) \cdot  \one\big\{ n^{-1/2} \le  \widetilde f_n(\theta) \le M\big\} 
        \]
        for $\theta\in[0,1]$. An application of Fubini shows that
        \begin{equation}
        \label{eq:variance_of_mahler_measure_after_reductions_and_fubini}
            \text{Var}\Big( \int_{[0,1]\setminus \mathcal{I}} X_n(\theta) \, {\rm d}\theta \, \Big) = \int_{[0,1]\setminus \mathcal{I}} \int_{[0,1]\setminus \mathcal{I}} \text{Cov}\Big(X_n(\theta) ,X_n(\varphi)\Big) \, {\rm d}\theta \, {\rm d}\varphi \, .
        \end{equation}
        As in the proof of Lemma~\ref{lemma:mahler_measure_of_gaussian_and_original_are_close}, we can integrate by parts and get that 
        \[
        \bE[X_n(\theta)] = (\log M)  \cdot \bP\big(\widetilde f_n(\theta) \le M\big) - \int_{n^{-1/2}}^{M}  \frac{\bP\big(\widetilde f_n(\theta) \le x\big)}{x}\, {\rm d}x \, ,
        \]
        and 
        \begin{align*}
            \bE[&X_n(\theta) X_n(\varphi)] \\ &= \int_{n^{-1/2}}^{M}\int_{n^{-1/2}}^{M} \log x \log y \, \frac{{\rm d}}{{\rm d} x} \, \frac{{\rm d}}{{\rm d} y}\Big( \bP\big(\widetilde f_n(\theta) \le x, \, \widetilde f_n(\varphi) \le y \big) \Big)  \\ &= \int_{n^{-1/2}}^{M} \log x \bigg[\log M \frac{{\rm d}}{{\rm d} x} \bP\big(\widetilde f_n(\theta) \le x, \, \widetilde f_n(\varphi) \le M \big) - \int_{n^{-1/2}}^{M} \frac{{\rm d}}{{\rm d} x} \bP\big(\widetilde f_n(\theta) \le x, \, \widetilde f_n(\varphi) \le y \big) \, \frac{{\rm d}y}{y} \bigg] \\ &= (\log M)^2 \cdot \bP\big(\widetilde f_n(\theta) \le M, \, \widetilde f_n(\varphi) \le M \big) - \log M \int_{n^{-1/2}}^M \frac{\bP\big(\widetilde f_n(\theta) \le x, \, \widetilde f_n(\varphi) \le M \big)}{x} \,  {{\rm d}x} \\ & \qquad \qquad  - \log M \int_{n^{-1/2}}^M \frac{\bP\big(\widetilde f_n(\theta) \le M, \, \widetilde f_n(\varphi) \le y \big)}{y} \,  {{\rm d}y} + \int_{n^{-1/2}}^M\int_{n^{-1/2}}^M \frac{\bP\big(\widetilde f_n(\theta) \le x, \, \widetilde f_n(\varphi) \le y \big)}{xy} \, {\rm d}x \, {\rm d}y \, .
        \end{align*}
        Note that all these computations apply equally well when the underlying polynomial is $\widetilde{g}_n$. Hence, Claim~\ref{claim:gaussian_comparison_CDF_on_unit_circle} allows us to compare the covariance to the covariance of independent Gaussians and observe that, for all $\theta,\varphi\in [0,1]\setminus \mathcal{I}$ such that $|\theta - \varphi|\ge n^{-1/2}$, we have
        \[
        \Big|\text{Cov}\Big(X_n(\theta) ,X_n(\varphi)\Big) \Big|\le \widetilde{O}(n^{-1/2}) \, .
        \]
        By Cauchy-Shwarz, we also have the trivial bound
        \[
        \Big|\text{Cov}\Big(X_n(\theta) ,X_n(\varphi)\Big) \Big|\le \widetilde{O}(1) 
        \, ,
        \]
        and by plugging into~\eqref{eq:variance_of_mahler_measure_after_reductions_and_fubini} we see that 
        \begin{align*}
            \text{Var}\Big(\int_{[0,1]\setminus \mathcal{I}} X_n(\theta) \, {\rm d}\theta \, \Big) & \leq \int_{[0,1]\setminus \mathcal{I}} \int_{[0,1]\setminus \mathcal{I}} \Big|\text{Cov}\Big(X_n(\theta) ,X_n(\varphi)\Big) \Big| \Big(\one_{\{|\theta-\varphi| \le n^{-1/2}\}} + \one_{\{|\theta-\varphi| > n^{-1/2}\}} \Big) \, {\rm d}\theta \, {\rm d}\varphi \\ &\le \widetilde{O} \bigg( \int_{0}^1\int_{0}^1  \one_{\{|\theta-\varphi| \le n^{-1/2}\}} {\rm d}x \, {\rm d}y  + n^{-1/2} \bigg) \\ &= \widetilde{O}(n^{-1/2}). 
        \end{align*}
        In view of~\eqref{eq:variance_mahler_measure_bound_after_reductions}, this shows that 
        \[
        \lim_{n\to \infty} \text{ Var}\Big( \int_{0}^1 \log \widetilde{f}_n(\theta) \cdot  \one\big\{ \widetilde f_n(\theta) \ge n^{-3/2}\big\} \, {\rm d}\theta \, \Big) = 0 \, ,
        \]
        as desired. 
    \end{proof}
    \begin{proof}[Proof of Proposition~\ref{prop:LLN_for_log_mahler_measure}]
        Recall the definition~\eqref{eq:def_of_event_B} of the event $\mathcal{B}$. As we want to prove convergence in probability, Claim~\ref{claim:no_root_on_unit_circle} implies that we can restrict to the complement event $\mathcal{B}^c$, on which 
        \[
        \int_{0}^1 \log \widetilde{f}_n(\theta) \, {\rm d}\theta  = \int_{0}^1 \log \widetilde{f}_n(\theta) \cdot  \one\big\{ \widetilde f_n(\theta) \ge n^{-3/2}\big\} \, {\rm d}\theta  \, .
        \]
        By Lemma~\ref{lemma:mahler_measure_variance_bound} and Chebyshev's inequality, the above converge in probability to its limiting expected value, which is shown to exists by a combination of Claim~\ref{claim:computation_of_expected_gaussian_mahler_measure} and Lemma~\ref{lemma:mahler_measure_of_gaussian_and_original_are_close}. Furthermore, we have 
        \[
        \lim_{n\to \infty} \bE\bigg[\one_{\mathcal{B}^c}\cdot \Big( \int_{0}^1 \log \widetilde{f}_n(\theta) \, {\rm d}\theta \Big)\bigg] = -\frac{\gamma}{2}
        \]
        and we are done. 
    \end{proof}

    \section{Gaussian comparison proofs}
    \label{sec:gaussian_comparison}
    \noindent
    The goal of this section is to prove Lemma~\ref{lemma:berry_esseen_in_annulus}, as a consequence of the Berry-Esseen theorem for random vectors (see, e.g.~\cite[Corollary~17.2]{Bhattacharya-Rao}). To do so, we will need to understand the covariance structure of the random vector
    \begin{equation}
    \label{eq:sample_two_points_polynomial_and_its_derivative}
    \Big( \frac{f_n(z)}{\sqrt{n}} , \frac{f_n^\prime(z)}{n^{3/2}} ,  \frac{f_n(w)}{\sqrt{n}} , \frac{f_n^\prime(w)}{n^{3/2}} \Big)    
    \end{equation}
    for $z,w\in \cA\setminus \mathcal{S}$ such that $|z-w|\ge n^{-1/2}$. In particular, as we identify $\bC^4\simeq \bR^8$, we need to lower bound the least singular value of the $8\times 8$ covariance matrix obtained by taking the real and imaginary parts of~\eqref{eq:sample_two_points_polynomial_and_its_derivative}. The relevant covariances are simple functions of trigonometric sums, and the next basic bound will be key in the analysis. 
    \begin{claim}
        \label{claim:bound_on_trig_sums}
        For $r\in[0,1]$, $\theta \in [n^{-1/2},\pi-n^{-1/2}]$ and $\ell\in\{0,1,2\}$ we have
        \[
        \qquad \bigg|\sum_{k=0}^n r^k k^\ell \cos(k\theta) \bigg| \le 8n^{1/2+\ell} \, , \qquad \text{and} \qquad \bigg|\sum_{k=0}^n r^k k^\ell \sin(k\theta) \bigg| \le 8n^{1/2+\ell} \, .
        \]
    \end{claim}
    \begin{proof}
        For $z=re^{\I\theta}$ we have
        \[
        \Big| \sum_{k=0}^{n} z^k\Big| = \Big| \frac{1-z^{n+1}}{1-z} \Big| \le \frac{2}{|1-z|} \le 4 \sqrt{n} \, .
        \]
        The cases $\ell =1$ and $\ell=2$ follow from the identities
        \[
        \sum_{k=0}^n k z^k = -\frac{nz^{n+1}}{1-z} + z\frac{1-z^{n+1}}{(1-z)^2} 
        \]
        and
        \[
        \sum_{k=0}^{n} k^2 z^k = -\frac{n^2 z^{n+1}}{1-z} -\frac{2nz^{n+1}}{(1-z)^2} + \frac{z(z+1)(1-z^n)}{(1-z)^3} \, ,
        \]
        together with the triangle inequality to yield the desired bound. 
    \end{proof}
    
    \noindent
    To study the covariance structure, we will use the following standard trigonometric identities:
    \begin{align}
        \label{eq:trig_identities}
        \nonumber
        2\cos(\theta)\cos(\varphi) &= \cos(\theta-\varphi) + \cos(\theta+\varphi) \\ 2\cos(\theta)\sin(\varphi) &= \sin(\theta+\varphi) - \sin(\theta-\varphi) \\ \nonumber 2\sin(\theta)\sin(\varphi) &= \cos(\theta-\varphi)-\cos(\theta+\varphi) \, .
    \end{align}
    Denote by $\displaystyle \sigma_{\min}(\Sigma) = \min_{\|v\|=1} \| \Sigma v\|$ the least singular value of a matrix $\Sigma$. We start by giving a lower bound on the least singular value of the covariance matrix evaluated at a single point $z\in \mathcal{A} \setminus \mathcal{S}$. 
    \begin{claim}
    \label{claim:lower_bound_on_least_singular_at_a_point}
        Let $z\in \mathcal{A}\setminus \mathcal{S}$ and set 
        \[
        \Sigma_z = \Cov\bigg( \frac{\text{\normalfont Re}\big(f_n(z)\big)}{\sqrt{n}} , \frac{\text{\normalfont Im}\big(f_n(z)\big)}{\sqrt{n}} , \frac{\text{\normalfont Re}\big(f_n^\prime(z)\big)}{n^{3/2}} , \frac{\text{\normalfont Im}\big(f_n^\prime(z)\big)}{n^{3/2}} \bigg)
        \]
        Then $\sigma_{\min}(\Sigma_z) \ge \widetilde\Omega(1)$.
    \end{claim}
    \begin{proof}
        Write $z=re^{\I\theta}$ and note that $z\in \cA\setminus \mathcal{S}$ implies that 
        \[
        r\in\Big[1-\frac{\log^3 n}{n},1\Big] 
        \]
        and $|\theta|\in[n^{-1/2},\pi-n^{-1/2}]$. Write
        \[
        \qquad s_\ell = \frac{1}{n^{1+\ell}} \sum_{k=0}^{n} r^{2k} k^\ell \, , \qquad \text{and} \qquad S_\ell = \int_0^{1}e^{-tb}t^\ell  \, {\rm d}t
        \]
        with $b= -2n\log(r)\in[0,\log^3n]$. The Euler-Maclaurin summation formula (used exactly as in the proof of Claim~\ref{claim:convergence_of_s_n}) implies that $|s_\ell - S_\ell| \le O(n^{-1/2})$ for $\ell\in\{0,1,2\}$. Combining this observation with Claim~\ref{claim:bound_on_trig_sums} and the trigonometric identities~\eqref{eq:trig_identities} implies that
        \begin{equation}
            \label{eq:asymptotic_for_sigma_z}
            \Sigma_z = \frac{1}{2}\begin{pmatrix}
                s_0 & 0 & 0& s_1 \\ 0& s_0 & s_1 & 0 \\ 0 & s_1 & s_2& 0 \\ s_1 & 0 & 0 & s_2
            \end{pmatrix} + O(n^{-1/2}) = \frac{1}{2}\begin{pmatrix}
                S_0 & 0 & 0& S_1 \\ 0& S_0 & S_1 & 0 \\ 0 & S_1 & S_2& 0 \\ S_1 & 0 & 0 & S_2
            \end{pmatrix} + O(n^{-1/2}) \, .
        \end{equation} 
        We further have
        \[
        \sigma_{\min} \left(\begin{pmatrix}
                S_0 & 0 & 0& S_1 \\ 0& S_0 & S_1 & 0 \\ 0 & S_1 & S_2& 0 \\ S_1 & 0 & 0 & S_2
            \end{pmatrix} \right) = \sigma_{\min}\bigg(\begin{pmatrix}
                S_0 & S_1 \\ S_1 & S_2
            \end{pmatrix}\bigg) \, .
        \]
        It remains to note that the operator norm of the above matrix is uniformly bounded, while 
         \[
        \det\begin{pmatrix}
            S_0 & S_1 \\ S_1 & S_2
        \end{pmatrix} = S_0 S_2 - (S_1)^2 = \frac{1+e^{-2b}-2e^{-b}-b^2e^{-b}}{b^4} \gtrsim \frac{1}{\log^{12}n} \, .
        \]
        We get that
        \[
        \sigma_{\min}\bigg(\begin{pmatrix}
                S_0 & S_1 \\ S_1 & S_2
            \end{pmatrix}\bigg) \ge \det\begin{pmatrix}
            S_0 & S_1 \\ S_1 & S_2
        \end{pmatrix} \cdot \bigg\| \begin{pmatrix}
            S_0 & S_1 \\ S_1 & S_2
        \end{pmatrix} \bigg\|_{\text{op}}^{-1} \,  \gtrsim \frac{1}{\log^{12}n}
        \]
        which, in view of~\eqref{eq:asymptotic_for_sigma_z}, gives the claim. 
    \end{proof}
    The next claim boosts the above to a similar least singular value bound for separated points.
    \begin{claim}
    \label{claim:least_singular_lower_bound_for_separated}
        For $z,w\in \cA\setminus \mathcal{S}$ such that $|z-w|\ge n^{-1/2}$, we set
        \[
        \Sigma_{z,w}  = \Cov \begin{pmatrix}
            \frac{\text{\normalfont Re}\big(f_n(z)\big)}{\sqrt{n}} , \frac{\text{\normalfont Im}\big(f_n(z)\big)}{\sqrt{n}} , \frac{\text{\normalfont Re}\big(f_n^\prime(z)\big)}{n^{3/2}} , \frac{\text{\normalfont Im}\big(f_n^\prime(z)\big)}{n^{3/2}} \\ \frac{\text{\normalfont Re}\big(f_n(w)\big)}{\sqrt{n}} , \frac{\text{\normalfont Im}\big(f_n(w)\big)}{\sqrt{n}} , \frac{\text{\normalfont Re}\big(f_n^\prime(w)\big)}{n^{3/2}} , \frac{\text{\normalfont Im}\big(f_n^\prime(w)\big)}{n^{3/2}}    
        \end{pmatrix}
        \]
        Then $\sigma_{\min}(\Sigma_{z,w}) \ge \widetilde\Omega(1)$.
    \end{claim}
    \begin{proof}
        Applying the trigonometric identities~\eqref{eq:trig_identities} together with Claim~\ref{claim:bound_on_trig_sums} shows that the covariance matrix $\Sigma_{z,w}$ asymptotically factors as
        \begin{equation}
        \label{eq:covariance_for_separated_asymptotically_factors}
                   \Sigma_{z,w} = \begin{pmatrix}
            \Sigma_z & 0 \\ 0 & \Sigma_w
         \end{pmatrix} + O(n^{-1/2}) \, ,
        \end{equation}
        and the claim follows at once from Claim~\ref{claim:lower_bound_on_least_singular_at_a_point}.
    \end{proof}
    The proof of the desired Gaussian comparison now easily follows. 
    \begin{proof}[Proof of Lemma~\ref{lemma:berry_esseen_in_annulus}]
        Both statements in the lemma follow immediately from the Berry-Esseen theorem for a sum of independent random vectors with uniformly bounded third moment, as it appears in~\cite[Corollary~17.2]{Bhattacharya-Rao}. For a pair of separated points $z$ and $w$ we apply Claim~\ref{claim:least_singular_lower_bound_for_separated} to show that the covariance matrix is non-singular (up to a tolerable $\polylog(n)$ loss), while for a single point $z\in \cA\setminus \mathcal{S}$ we apply Claim~\ref{claim:lower_bound_on_least_singular_at_a_point} instead. 
    \end{proof}
    We conclude with a proof that we owe from Section~\ref{sec:lln_for_sum_in_annulus}. 
    \begin{proof}[Proof of Claim~\ref{claim:decorrelation_in_annulus}]
        By Lemma~\ref{lemma:comparison_with_a_gaussian_on_the_event_C_pm} item~\eqref{eq:lemma_comparison_with_a_gaussian_on_the_event_C_pm_third}, it suffices to prove the claim for the Gaussian random polynomial $g_n$. Indeed, the corresponding statement for Gaussians follows at once from~\eqref{eq:covariance_for_separated_asymptotically_factors}, as it shows that the difference between the joint density of
        \[
        \Big( \frac{g_n(z)}{\sqrt{n}} , \frac{g_n^\prime(z)}{n^{3/2}} ,  \frac{g_n(w)}{\sqrt{n}} , \frac{g_n^\prime(w)}{n^{3/2}} \Big)
        \]
        and the product of the densities of
        \[
        \qquad \Big( \frac{g_n(z)}{\sqrt{n}} , \frac{g_n^\prime(z)}{n^{3/2}} \Big) \qquad \text{and}\qquad  \Big( \frac{g_n(w)}{\sqrt{n}} , \frac{g_n^\prime(w)}{n^{3/2}} \Big)
        \]
        is $O(n^{-1/2})$. This shows that the difference of expectations is $\widetilde{O}(n^{-1/2})$ and we are done. 
    \end{proof}
    
	\printbibliography[heading=bibliography]
	
	\medskip
    \medskip
    
\end{document}